\documentclass[12pt]{article}
\usepackage{amsmath,amsthm,amssymb,latexsym,color}
\usepackage{bbm}

\usepackage{hyperref}

\let\oldproofname=\proofname
\renewcommand{\proofname}{\rm\bf{\oldproofname}}

\title{On random diameters of  convex bodies}
\author{O. Gu\'edon  \and A. E. Litvak \and K. Tatarko \and B.-H. Vritsiou
}

\date{ }

\newcommand\address{\noindent\leavevmode
	
	\medskip

	\noindent
	O. Gu\'{e}don \\
	Univ Gustave Eiffel, Univ Paris Est Creteil, \\
	CNRS, LAMA UMR8050 F-77447 Marne-la-Vallée, France\\
	\texttt{\small%
		e-mail:  olivier.guedon@univ-eiffel.fr}
	
	\medskip

	\noindent
	A.E. Litvak
	and B.-H. Vritsiou\\
	Department of Mathematical~and Statistical~Sciences,\\
	University of Alberta, \\
	Edmonton, AB, T6G 2N8, Canada\\
	\texttt{\small
		e-mails:  aelitvak@gmail.com \, \, and \, \, vritsiou@ualberta.ca
		}
		
		\medskip
		
		\noindent
		K. Tatarko \\
		Department of Pure Mathematics, \\
	 	University of Waterloo, \\
		Waterloo, ON, N2L 3G1, Canada \\
		\texttt{\small
		e-mail: ktatarko@uwaterloo.ca }
}

\newcommand{\RR}{\mathbb{R}}
\newcommand{\N}{\mathbb{N}}
\newcommand{\Rn}{\RR^n}
\newcommand{\RN}{\RR^{n+1}}
\newcommand{\RNN}{\RR^{N}}
\newcommand{\Rk}{\RR^k}
\newcommand{\Rm}{\RR^m}

\newcommand{\E}{\mathbb{E}}

\newcommand{\Gn}{\Gamma_n}

\newcommand{\Gkn}{\Gamma_{kn}}

\newcommand{\Id}{\mathrm{Id}}
\newcommand{\im}{\mathrm{Im \,}}
\newcommand{\rad}{R}
\newcommand{\eps}{\varepsilon}

\renewcommand{\P}{\mathbb{P}}

\newtheorem{theorem}{Theorem}[section]
\newtheorem{lemma}[theorem]{Lemma}
\newtheorem{remark}[theorem]{Remark}

\newtheorem{proposition}[theorem]{Proposition}

\newtheorem{corollary}[theorem]{Corollary}

\begin{document}

\maketitle

\begin{abstract}
 Let $K \subset \RNN$ be a convex body containing the origin in its interior. In this work, we study the diameters of random sections of $K$ and derive upper and lower bounds for them in terms of geometric parameters of $K$. Our bounds hold with large probability, and they offer new insights into this widely studied subject. Our upper bound complements the so-called low $M^*$-estimate and in many cases it is much sharper. The two lower bounds that we give are each of a different nature: depending on the body in question each time, either could be better, and in many interesting cases it matches the upper bound too. Subsequently, we apply our results to determine random diameters of $p$-ellipsoids (images of $\ell_p$ balls under diagonal operators), improving upon previously known results and achieving sharp estimates in many cases. One notable application is to Information-Based Complexity Theory, where we manage to establish a  simple (and essentially optimal) dichotomy in response to a very natural conjecture posed by Hinrichs, Prochno and Sonnleitner in 2023. Our solution settles precisely when it is useful to replace the optimal information used for the recovery of vectors from a $p$-ellipsoid with random (Gaussian) information, which can be more practical to obtain.
\end{abstract}

\bigskip

{\small

\section{Introduction}
\label{sec:introduction}
Our study is motivated by questions that naturally arise
in Asymptotic Geometric Analysis (see e.g. \cite{MS-1200, Pisier1989, NTJ-book, AGM})
and in Information Theory, particularly in Information-Based
Complexity (IBC) Theory, as  presented in \cite{HKNPU} and \cite{HPS}
(see also the books \cite{NW-1, NW-2, NW-3}  for a comprehensive introduction to the theory).
Let $K$ be a convex body in $\RNN$ with the origin
in its interior. In IBC theory, a  central goal is to recover a vector $x \in K$ from given data $T_n(x) \in \RR^n$  obtained as the output
of  a linear mapping $T_n : \RNN \to \RR^n$, where $n < N$. The recovery algorithm must be optimized, and
the error is measured in the Euclidean norm. It is known (see e.g.
\cite[Theorem~4.2]{NW-1}) that the worst case error of the best recovery algorithm  is determined by the radius of
$K \cap \ker T_n$. This raises the following questions: What happens if the data comes from a random Gaussian mapping instead of a fixed linear map $T_n$?
Can the radius of random information decay at a rate comparable to that of optimal information - or at least sufficiently fast to zero, uniformly in $N$ and as $n$ grows? Or is it bounded from below?
In the papers \cite{HKNPU, HPS}, the authors addressed these natural questions for ellipsoids and $p$-ellipsoids
(i.e., images of standard $\ell_p$-balls under diagonal transformations). For a broader context and related problems,
we also refer to the two recent surveys \cite{HKNPU-Survey, KU2026-ActaNumerica}.

On the other hand, the study of the diameter of best (or of random) sections of convex bodies corresponds to
Gelfand numbers (respectively, random Gelfand numbers). This topic has a long history in  Asymptotic Geometric Analysis and  Approximation Theory,
dating back to Dvoretzky's theorem. Since Dvoretzky's theorem guarantees the existence of a Euclidean section
of a convex body, it gives something stronger than merely bounding the diameter. As a result, it only provides  bounds for relatively low dimensional sections, with dimension up to an important parameter of the convex body called the Dvoretzky dimension.
For truly high-dimensional sections (i.e. those of small codimension), the standard approach relies on the
so-called Milman's {\it ``low $M^*$-estimate"} \cite{M1985, PajorTomczak, Gordon1988} (see Remark~\ref{M*}), which
can yield  sharp bounds when the body $K$ is in a specific position \cite{GianMil1997}.
Nevertheless, as already observed by Giannopoulos and Milman  (see Example 2.2 in \cite{GianMil1997}), for an ellipsoid with highly incomparable lengths of semi-axes,
this method fails to provide  any non-trivial bound for the diameter of random sections.

The goal of this paper is twofold. First, we address the general case of arbitrary convex bodies.
We develop a novel, general approach that is significantly more sophisticated than the
low $M^*$-estimate. It is partially inspired by \cite{HKNPU}, where the authors
investigated the random diameters of  ellipsoids. Additionally, we  provide two distinct lower bounds,
which are incomparable and complement each other.
Second, we turn to important applications in IBC Theory, focusing on the case of
 $p$-ellipsoids defined as images of $\ell_p$-balls under diagonal operators with $\sigma = \{ \sigma_i \}_{i \ge 1}$
 on the diagonal (see definition in \eqref{def:Ep} below).
We reveal a clear dichotomy, depending on the lengths $\sigma_i$ of the semi-axes of
the $p$-ellipsoid. Our results are derived in a finite dimensional setting, meaning we always work in
$\RR^N$. By letting $N \to \infty$, we establish that
if $\sigma \in \ell_{p'}$ (where $p'$ is the conjugate exponent of $p$) then random information is useful and almost as effective as the optimal information,
while random information is useless if $\sigma \notin \ell_{p'}$. This answers a question raised in \cite{HPS}.
We finally consider several examples depending on the decay of the sequence $\sigma$:
\begin{itemize}
	\item[--] for the case of exponential decay, we answer the aforementioned Giannopoulos--Milman question from \cite{GianMil1997} in the setting where the
	semi-axes have highly incomparable lengths,
	\item[--] for the case of polynomial decay, we provide complete answers to all the questions raised in \cite{HKNPU, HPS} regarding the behavior of the decay of random radii.
\end{itemize}
In these cases, our upper and lower bounds are sharp with respect to the dependencies on both the section's codimension $n$
and the ambient space dimension $N$. It is interesting to note that our upper bound, Theorem~\ref{MainThm} complements
the above mentioned low $M^*$-estimate (see Remark~\ref{M*}). The two bounds are incomparable and both need to be used, depending in each case on the behavior of the axis lengths of the ellipsoid. Similarly, we have two incomparable lower bounds, Theorems~\ref{LowerThm}   and \ref{LowerGeneralThm}, and again both need to be considered, depending on the behavior of the semi-axes, in order to obtain sharp results.

\medskip

We now formally define the random radius and discuss immediate bounds derived from  Dvoretzky's theorem.
Note that it is more convenient to work with the radius of a body (the smallest radius of the Euclidean ball
containing the body) rather than with its diameter, but this is without loss of generality as the two are equivalent up to a factor of 2.
The $n$-th random radius of a convex body $K\subset \RNN$ which contains the origin, denoted by ${\cal R}_n(K)$, is
defined as the radius of the section of $K$ by an $n$-codimensional subspace of $\RNN$ uniformly distributed  with respect to the
 normalized Haar measure on the Grassmannian of all $n$-codimensional subspaces
  of $\RNN$. In other words, ${\cal R}_n(K)$
 is the random variable defined by
\begin{equation}
	\label{eq:def}
{\cal R}_n(K) = \max_{y \in E_n \cap K} |y|,
\end{equation}
where $E_n$ is a random $n$-codimensional subspace of $\RNN$ and $|y|$ denotes  the Euclidean norm of $y$.

In Asymptotic Geometric Analysis, random radii play a central role in the study of properties of high-dimensional convex
bodies, and of their sections and projections.
They naturally come up in Milman's celebrated version of the  Dvoretzky theorem \cite{M-Dvor} (note that Milman's proof works in the
non-symmetric case as well; see also \cite{Gordon1985, Gordon-ann-88} for a proof using Gaussian processes which covers the non-symmetric case too).
This theorem  states that for every convex body $K$ containing the origin in its interior, a random section $K\cap E$
of dimension  $d_K =   (\ell(K)/(2b_K))^2$  is almost Euclidean with probability at least
$1-4\exp(-d_K/16)$, that is,
\begin{equation}
	\label{DvorEstimate}
     \frac{\sqrt{N}}{8\ell(K) } B_2^N  \cap E \subset K\cap E \subset \frac{8\sqrt{N}}{ \ell(K)} B_2^N  \cap E,
\end{equation}
where $B_2^N$ is the canonical Euclidean ball in $\RNN$, $1/b_K$ is the inradius of $K$ (i.e., $b_K$ is the minimal number
 satisfying $(1/b_K)  B_2^N\subset K$,
and also coincides with the Lispchitz constant of $\|\cdot\|_K$),  and  $\ell(K)$ is the expectation of $\|G\|_K$ for a
standard Gaussian vector in $\RNN$ (see the next section for precise definitions). The probability bound follows
from \cite{Gordon1985, Gordon-ann-88} combined with the
Gaussian concentration inequality~(\ref{ineq:concentration}).
In particular, with the aforementioned probability, the random radius satisfies
\begin{equation}
	\label{DvorEstimate2}
   \frac{\sqrt{N}}{8\ell(K) }\leq {\cal R}_n(K)  \leq \frac{8\sqrt{N}}{ \ell(K)}
\end{equation}
whenever $n\geq N-d_K$.
We would also like to note that $\ell(K)$ has a clear geometric
meaning. It can be expressed as $\ell(K)=c_N\sqrt{N}M_K$ where $c_N\in (0, 1)$, $c_N\to 1$, and $M_K$ is the average of  $\|\cdot\|_K$
on the unit Euclidean sphere $S^{N-1}$ with respect to the normalized Haar measure.
Observe that the radius of $K \cap E$ measures the size of the smallest Euclidean ball containing the set $K \cap E$, focusing only on the right hand side
of \eqref{DvorEstimate}.
 By leveraging this observation, Klartag and Vershynin \cite{Klartag-Vershynin-paper} managed to extend the validity of the upper bound in \eqref{DvorEstimate2} for random sections of dimension up to another
parameter which can be sometimes, depending on the position of the convex body, significantly larger than $d_K$ (see also  \cite{Paouris-Tikhomirov-Valettas-paper}).
Note however that the bounds on ${\cal R}_n(K)$ remain of the same form, that is, they are of the order of $\sqrt{N}/\ell(K)$.
Thus, as can also be confirmed by our results, in general these estimates can only apply to a restricted range of really large codimensions.

A key challenge is
to understand how the radius of a section behaves when the subspace $E$ has smaller codimension and when the convex body is not in any particular position.
The smallest possible radius of an $n$-codimensional section of $K$ by a subspace $E$ corresponds to the $(n+1)$-th Gelfand number of the identity operator $I : (\RNN, \|\cdot\|_K) \to \ell_2^N$, denoted by $c_{n+1}(K)$.
These numbers have been extensively studied in both Asymptotic Geometric Analysis and Approximation Theory
\cite{Pisier1989, Carl-book}. Since the best section is typically unavailable, the standard approach is to estimate
${\cal R}_n(K)$ and demonstrate that the random radius behaves similarly to the best one whenever possible.
This approach has led to  sharp results in Approximation Theory, such as estimates of
Gelfand numbers for the unit ball of $\ell_1^N$ \cite{K77, GG1984}.  In Asymptotic Geometric Analysis, an approach based on the low $M^*$-estimate was developed
 in \cite{M1985, PajorTomczak, Gordon1988} and random radii were intensively studied  in several works (see e.g.
 \cite{GiaM1998, LT2000, LitvPajorTom, EmMil}). Moreover, a phenomenon that could be coined as {\it deterministic implies random} was
 discovered, in the sense that the random radius cannot be significantly worse than the best one. Following \cite{LT2000},
we define a certain quantile of ${\cal R}_n(K)$  as {\it a random Gelfand number}, denoted by $cr_{n+1}(K)$.
An immediate consequence of Theorem~3.2 in \cite{LT2000}  is that for some absolute constants
$C_0, C_1 >1$ and any $\mu:=n/k>C_0$,  one has
\begin{equation} \label{detimran-1}
 cr_{n+1}(K)\leq C_1 \sqrt{N/n}\, c_{k+1}(K)
 \end{equation}
(and the result holds in the non-symmetric case as well). Later in \cite{LitvPajorTom} it was shown
that in the centrally-symmetric case  one can take $\mu >1$ at the cost of a larger factor in front of $c_{k+1}(K)$ which becomes
\begin{equation} \label{detimran-2}
 cr_{n+1}(K)\leq C_1   \sqrt{N/n}^{1+1/(\mu-1)} \, c_{k+1}(K) .
\end{equation}
 Note that the dimension $N$ of the ambient space appears and the factor
$\sqrt{N/ n}$ can be very large for small codimensions. The primary goal of this work is to address the case of
 small codimensions.  Our lower bounds demonstrate that, in various cases, a random section can yield
 a worse estimate than the optimal one, showing that the factor $\sqrt{N/n}$ may be unavoidable. This is illustrated
 in Section~\ref{sec:example}.

\smallskip
We now present our main general results, which provide, with high probability,  upper and  lower bounds for the random radius
of an $n$-codimensional section of a convex body $K$. Here, we use the standard notation $\ell_*(K)$ for $\ell(K^\circ)$.

\begin{theorem}
	\label{MainThm}
There exist universal constants $C\ge 2$ and $0< \gamma  < 1$ such that for every
$0 < \eps \le 1/(2C)$ the following holds.
Let $1\le n \le N$ and  $K \subset \RNN$ be a convex body containing the origin in its interior.
Then for every $1 \le k \le n$ and  every subspace $E$  of codimension $k$, one has
\[
{\cal R}_n(K) \le
\left(  \frac{6 \sqrt{\log(1/\eps)}}{\eps} \sqrt{\frac{2n}{n-k+1}}\right)  \rad(P_E K) + \frac{2 \sqrt n}{\eps (n-k+1)} \
\ell_*(P_E K )
\]	
with probability at least
\[
1 - 3 (C \eps)^{n-k+1} - k e^{- \gamma n}.
\]
\end{theorem}

\begin{remark}
In  Theorem \ref{MainThm} (as well as in  Corollary~\ref{cor:lp-ellipsoid} below),
it is natural to minimize the upper bound over all choices of the parameter $k$.
Let $k_0$ and $E_0$ be chosen to minimize the deterministic quantity
\[
\left(  \frac{6 \sqrt{\log(1/\eps)}}{\eps} \sqrt{\frac{2n}{n-k+1}}\right)  \rad(P_E K) +  \frac{2 \sqrt n}{\eps (n-k+1)} \ \ell((P_E K )^\circ)
\]
over every $1 \le k \le n$ and every $E$ of codimension $k$. Then ${\cal R}_n(K)$ is bounded by the corresponding
quantity for $k_0$ and $E_0$ with probability at least
\[
1 - 3  (C \eps)^{n-k_0+1} - k_0 e^{- \gamma n}.
\]
Note also that the option $k_0=n$ is a possible choice that leads to a probability greater than a fixed parameter
$1-\delta$ (for $n$ sufficiently large, and by choosing $\eps$ small enough).
\end{remark}

\begin{remark}\label{M*}  Theorem~\ref{MainThm} should be compared with a standard tool in Asymptotic Geometric Analysis for
bounding random radii --- the  low $M^*$-estimate" \cite{M1985, PajorTomczak},  which states that, for every
 $\lambda \in (0,1)$ and $n=\lambda N$, with high probability one has ${\cal R}_n(K) \leq 2 \ell^*(K)/\sqrt{n}$
(for precise constants and probability estimates we refer to \cite{Gordon1988}, where the Gaussian approach was used).
As our examples in Section~\ref{sec:example}
show, the two bounds are incomparable and, depending on the convex body in question each time, either of them could be sharp. In particular, in the case
of an ellipsoid with exponential decay of axis lengths, Theorem~\ref{MainThm} provides a sharp bound, while the  low $M^*$-estimate
leads to a very weak estimate, the latter fact already having been observed in Example~2.2 of \cite{GianMil1997}.
\end{remark}

We would like to mention an immediate simple corollary of this theorem. It is well known (and not difficult
to check) that $\ell((P_E K )^\circ)\leq  \sqrt{N} \, \rad(P_E K)$. We also use the following notation:
given a convex body $K$ containing the origin in its interior and a subspace $E\subset \RNN$, we define
$$
  \delta_{E,K} := \left\|  P_E :\, (\RNN, \|\cdot\|_K)\to (E, \|\cdot\|_K) \right\|
  =\inf\{\lambda >0\, |\, P_E K \subset \lambda K\cap E\}.
$$
In particular,
for all $x \in E$, one has $\|x\|_K \le \delta_{E,K} \|x \|_{P_E K}$.

\begin{corollary}
	\label{cor:detran}
There exist universal constants $C\ge 2$ and $0< \gamma  < 1$ such that for every
$0 < \eps \le 1/(2C)$ the following holds. Let $1 \le k \le n$ and $E$ be a $k$-codimensional
subspace minimizing $\rad (K\cap E)$, i.e., $c_{k+1}(K)= \rad (K\cap E)$.
Then
\[
{\cal R}_n(K) \le
  \left( \frac{12 \sqrt{\log(1/\eps)}}{\eps} \, \frac{ \sqrt{N  \, n}}{n-k+1}\right)  \delta_{E,K}\,  c_{k+1}(K)
\]	
with probability at least
\[
1 - 3 (C \eps)^{n-k+1} - k e^{- \gamma n}.
\]
In particular, with $\mu = n/k$, for some absolute constant $C_0$ we have
\[
cr_{n +1} (K) \le   \frac{C_0\, \mu }{\mu-1}\,
  \sqrt{\frac{N}{n}}\,   \delta_{E,K}\,  c_{k+1}(K).
\]	
\end{corollary}

The ``in particular'' part of this corollary should be compared with (\ref{detimran-1}) and (\ref{detimran-2}).
In the case when $\delta_{E,K}$ is bounded (which e.g. is the case for $p$-ellipsoids with $p\geq 2$),
this bound improves the previous estimates.

\medskip

We now turn to lower bounds.

\begin{theorem}
	\label{LowerThm}
Let $1\le n < N$ and  $K \subset \RNN$ be a  convex body containing the origin in its interior.
Then for every subspace $F$ of dimension $n+1$ and for every $t > 0$,
\[
{\cal R}_n(K) \ge \frac{\sqrt{n}}{2 \left(\ell({K \cap F}) + t \ \rad((K \cap F)^\circ)\right)}
\]
with probability at least
\[
1 - e^{-n/8} - e^{-t^2/2}.
\]
\end{theorem}

As above, we relate this bound to the Dvoretzky theorem. Note that, for any convex body
$L\subset \Rm$ that contains the origin in its interior,  $b_L$  equals the radius $R(L^\circ)$, and that $b_L\leq \ell({L})\leq \sqrt{m}\, b_L$. Therefore, by choosing
$t= s \, \ell({K \cap F})/ b_{K \cap F}$ in Theorem~\ref{LowerThm}, we obtain the following corollary.

\begin{corollary}
	\label{cor:low-dvor}
Let $1\le n < N$ and  $K \subset \RNN$ be a  convex body containing the origin in its interior.
Then for every subspace $F$ of dimension $n+1$ and for every $s \geq 1$,
\[
{\cal R}_n(K) \ge \frac{\sqrt{n}}{4 \, s \, \ell({K \cap F})}
\]
with probability at least
\[
1 - e^{-n/8} - \exp{ \left( - s^2 \ell({K \cap F})^2/(2\,b_{K \cap F}^2) \right)}.
\]
\end{corollary}

\begin{remark}
 As mentioned above, the parameter $M_{K \cap F}$ is the average of the norm $\|\cdot\|_K$
 over the unit Euclidean sphere in $F$  (recall that the dimension of $F$ is $n+1$). Therefore
  $1/M_{K \cap F}\approx \sqrt{n}/ \ell({K \cap F})$, which is precisely the order of the radius provided by
 Dvoretzky's theorem when applied to $K \cap F$ (see (\ref{DvorEstimate})). Consequently, Corollary~\ref{cor:low-dvor} shows that
for a random $n$-codimensional section of $K$, the diameter  cannot decrease below the radius of the Euclidean ball furnished by
Dvoretzky's Theorem for the $(n+1)$-dimensional convex body $K \cap F$. It is worth emphasizing  that this conclusion  holds for every
$(n+1)$-dimensional  subspace $F$. Moreover, the ``Dvoretzky radius" of $K \cap F$ may be substantially  larger than
 that of $K$ itself, as illustrated by  an ellipsoid with $N-n-1$ short principal axes.
\end{remark}

Next we provide a lower bound of a different type, in which we will use the  notation $\delta_{E,K}$ introduced
before Corollary~\ref{cor:detran}. Observe that, in contrast with Theorem \ref{MainThm} and Corollary ~\ref{cor:detran}, in the next theorem we do not need the additional assumption that $k\leq n$.

\begin{theorem}
	\label{LowerGeneralThm}
Let $1 \le n < N$ and $K \subset \RR^N$ be a convex body containing the origin in its interior.
Assume that $1 \le k < N$ is such that there exists  a $k$-codimensional subspace $E$ satisfying the following hypothesis:
	\begin{equation}
		\label{assumption:key}
		\ell_*(P_E K) \ge 6 \sqrt{n}\, R(P_E K).
	\end{equation}
Then, with probability at least  $1 - 2e^{- 2 n}$, one has
\[
    {\cal R}_n(K) \ge
    \left( \frac{1}{R(K \cap E^\perp)} +  \frac{6 \, \delta_{E,K} \, \sqrt{n}}{\ell_*(P_E K)} \right)^{-1}
    \ge \frac{1}{2} \min \left\{ R(K \cap E^\perp) ,\,  \frac{\ell_*(P_E K)}{6 \, \delta_{E,K} \, \sqrt{n}} \right\}.
\]
\end{theorem}

\begin{remark} \label{lower-nat}
We would like to emphasize that our two lower bounds, Theorem~\ref{LowerThm} and Theorem~\ref{LowerGeneralThm}, are of a different nature: the former one works with an $(n+1)$-dimensional subspace (usually applied with a ``large" part of the body, as will also be illustrated by the application to ellipsoids in Section~\ref{sect-p-ell}; see Remark~\ref{ill-ell}), while the latter one uses a $k$-codimensional subspace (usually applied with a ``small" part of the body). The concrete examples of $p$-ellipsoids considered in the sequel will confirm that either of the two bounds could be sharp, and matching the corresponding upper bound, in different cases.
\end{remark}

In Section~\ref{sect-p-ell} we apply our results to derive
estimates for the radii of random sections of $p$-ellipsoids, a problem that was studied in particular in \cite{HKNPU, HPS}.
Our bounds are expressed in terms of the lengths $\{\sigma_i\}_{i=1}^N$ of the semi-axes of the
$p$-ellipsoid ${\cal E}_p$ (see \eqref{def:Ep} for the definition).

\smallskip

We conclude by  mentioning another consequence of our results.
Letting the  dimension $N$ of the ambient space tend to infinity, as is often required in Approximation Theory, we obtain the following
 general dichotomy, which provides a satisfactory characterization of when Gaussian measurements may help
  in  IBC theory.
While the corresponding upper bound (established in \cite{HPS}) follows from the so-called ``low $M^*$-estimate," our results yield a sharp lower bound, and resolves the general open problem posed in \cite{HPS}. Recall that for  $p \in (1,  \infty)$, the conjugate exponent  $p'$ is defined
	by $1/p + 1/p' = 1$.

\begin{theorem}
	\label{cor:dichotomy}
 Let $p \in (1,  \infty]$ and $\{\sigma_j\}_{j=1}^\infty$ be a sequence of  non-increasing non-negative numbers.
 Then the following dichotomy holds.
	\begin{enumerate}
		\item If $ \sum_j \sigma_j^{p'}$ converges, then there exists a sequence $\{\varepsilon_n\}_{n =1}^\infty$ such that
		$\lim_{n \to \infty} \varepsilon_n = 0$ and
		\[
		\mathit{for \ all \ } N \in \N, \ \mathit{for \ all \ }  1 \le n < N, \quad \mathit{one \ has } \quad \sqrt{n} \, {\cal R}_n({\cal E}_p) \le \varepsilon_n.
		\]
		In this case, the optimal information and the Gaussian measurements are not very different from the point of
view of IBC theory.
		\item If $\sum_j \sigma_j^{p'}$ diverges, then there exists a function $f$ with
		 $\lim_{N \to \infty} f(N) =  \infty$ such that
		\[
		\mathit{for \ all \ } N \in \N, \ \mathit{for \ all \ } 1 \le n \le f(N),  \quad\quad \mathit{one \ has } \quad \P \left( {\cal R}_n({\cal E}_p) \ge  \sigma_{1} / 2 \right) \ge 1 - 2e^{-2n}.
		\]
		In this case,  Gaussian measurements cannot help at all in IBC Theory.
	\end{enumerate}
\end{theorem}

\begin{remark}
	In the case $p=1$, a similar dichotomy holds, however the discussion is based on the behavior of
	$\sup_{j \ge 1} \sigma_j \sqrt{\log(1+j)}$ instead of $\sum_j \sigma_j^{p'}$.
\end{remark}

\begin{remark}
Such a dichotomy already appeared in the  Approximation Theory literature in the  Hilbert space setting, that is, when $p=2$
(see e.g. \cite{HKNPU} and Chapter~26 in \cite{NW-3}).  Our approach yields more precise estimates.
\end{remark}

In  Section~\ref{sec:example}, we discuss several examples for which we obtain sharp estimates in all parameters $n$ and $N$.
When the sequence
$\sigma=\{\sigma_j\}_{j=1}^N$ has exponential decay, namely $\sigma_j = q^{-j}$ for a fixed $q\in (0,1)$, we prove in Proposition \ref{caseGeom}
that, with high probability, the random radius of an $n$-codimensional section of a
$p$-ellipsoid is  of  order  $\sqrt{n} \, \sigma_{n+1}$. This resolves a problem left open
by Giannopoulos and Milman in \cite{GianMil1997}, where it was shown that approaches based
on the low $M^*$-estimate are suboptimal (see also \cite{GiaMT2005} for lower bounds and
\cite[Corollary~9]{HKNPU} for an upper bound of  order $n^2 \, \sigma_{n+1}$).
Furthermore, when the sequence
$\sigma=\{\sigma_j\}_{j=1}^N$  exhibits polynomial decay, namely  $\sigma_j = j^{-\alpha} (\log(j+1))^{-\beta}$,
we obtain in Propositions~\ref{caseI}, \ref{caseIbis}, \ref{caseII}--\ref{caseIV} sharp estimates for the
random diameter ${\cal R}_n({\cal E}_p)$. This completely resolves  a question raised by Hinrichs,
Prochno and Sonnleitner \cite{HPS}. In that work,
only the case $\beta = 0$ was considered (see Fig.~1 and the Conjecture on page~7 in \cite{HPS}).

\smallskip

The paper is organized as follows. In
Section~\ref{sec:notations}, we introduce  the notation used throughout the paper.
 Section~\ref{sect-p-ell} is devoted to the special case of $p$-ellipsoids, where we present complete statements
 in terms of the lengths of the semi-axes. The proofs of the
 main theorems are given in  Section~\ref{sec:upper}, while the proofs of the corresponding results for
 $p$-ellipsoids are deferred to  Section~\ref{sec:ellipsoid}.
 Finally, in Section~\ref{sec:example}, we analyze several examples in detail   and prove Theorem~\ref{cor:dichotomy}.
%
%
%
%

\section{Notation and preliminaries}
\label{sec:notations}

Let $N$ be a fixed integer, $\RNN$ be the ambient space and $| \cdot |$ be the canonical Euclidean norm.
A compact convex set $K \subset \RNN$ containing the origin in its interior  will
 be called a  convex body. We denote by
 $$
   \|\cdot\|_K:=\inf \{\lambda >0 \,\, | \,\, x\in \lambda K\}
 $$
 the gauge of $K$ (then $K$ is the unit ball of $\|\cdot\|_K$), and by
 $K^\circ = \{y\,\, |\,\,  \forall x \in K, \langle x, y \rangle \le 1 \}$ the polar body  of $K$.
The radius of $K$ (which is one half of the diameter of $K$ whenever $K$ is symmetric) is denoted by
\[
R(K):=\max_{y \in K} |y|,
\]
meaning equivalently that $R(K)$ is the smallest number such that
$K  \subset R(K)  B_2^N$,
where $B_2^N$ is the Euclidean unit ball of the ambient space.
Clearly, $\mbox{diam}(K)/2\leq  R(K)\leq \mbox{diam}(K)$, where
$\mbox{diam}(K)$ denotes the diameter of $K$.
Given a subspace $E$ of $\RNN$, we denote the orthogonal projection onto $E$ by $P_E$.
When $E=\Rk$ we simply write $P_k$ (thus $P_k$ is the coordinate projection onto the first $k$ coordinates).
As is mentioned in the introduction, given a subspace $E\subset \RNN$,
by $\delta_{E,K}$ we denote the minimal number  such that $P_E K \subset \delta_{E,K} K\cap E$.

 Below $g_i$, $g_{ij}$,  $i\geq 1$, $j \ge 1$ always denote a family
of i.i.d. standard ${\cal N}(0,1)$ random variables. We often
consider a standard Gaussian vector in $\RR^m$, denoted by
$$G^{(m)} = (g_1, \ldots, g_m).$$

Given a convex body $K\subset \RNN$ containing the origin in its interior, we define
the parameters $\ell(K)$ and $\ell_*(K)$  by
\[
\ell(K) = \E \|G^{(N)}\|_K \quad \quad \mbox{ and } \quad \quad \ell_*(K)=\ell(K^\circ).
\]
More generally, given a non-empty set $T\subset \RNN$, we set
$$
  \ell_*(T) = \E \sup_{x\in T}\, \langle G^{(N)}, x \rangle.
$$
Note that  $\ell_*(T):=\ell(T_0^\circ)$,
where $T_0$ is the convex hull of $T$ and $\{0\}$, and that $\ell_*(T)$ is often called
the Gaussian complexity of $T$.  As we explained in the introduction, the following
parameter, called Dvoretzky's dimension, plays an important role in Asymptotic Geometric Analysis:
$$
  d_K =  \frac{1}{4} \left(\frac{\ell(K)}{b_K}\right)^2,
$$
where  we recall that $b_K= \max\limits_{x\in S^{N-1}} \|x\|_K = R(K^\circ)$.

We write $\Gn = (g_{ij})_{1 \le i \le n, 1 \le j \le N}$ and consider it as
the random Gaussian operator from $\RNN$ to $\Rn$.
 For every integer $1\le k \le N$ we denote  by $\Gkn : \Rk \to \Rn$ the restriction of $\Gn$ to $\Rk$.
Sometimes it will be convenient to identify $\Gkn$ with an $n\times N$ matrix whose last $N-k$ columns are identically zero (whereas the entries in its first $k$ columns are i.i.d. standard Gaussian random variables, as before).
Then we set  ${\Gamma_n^k} := \Gn - \Gkn$, that is,
\begin{equation}
	\label{def:Gamma-k}
	\Gamma_n^k
	=
	\left( 0_n, \ldots, 0 _n, G_{k+1}, \ldots, G_N
	\right),
\end{equation}
where $0_n$ denotes the zero vector in $\RR^n$ and
 for all $j >k$, $G_j = (g_{ij})_{1 \le i \le n}$ is a standard Gaussian vector in $\RR^n$.
It is well known and follows from the rotational invariance of the Gaussian measure that $\ker \Gn$ has the same distribution as a random $n$-codimensional subspace $E\subset \RNN$ with respect to the normalized Haar measure on the Grassmannian.

We will often apply  the following Gaussian concentration inequality, see e.g. \cite[Proposition~2.18]{Ledoux}.
Let $F : \RR^d \to \RR$ be an $L$-Lipschitz function with respect to the Euclidean norm in $\RR^d$.
 Then for every $t > 0$,
\begin{equation}
	\label{ineq:concentration}
	\max \Big\{ \P \left(F(G^{(d)}) - \E F(G^{(d)}) > t \right)  , \P \left(F(G^{(d)}) - \E F(G^{(d)}) < -  t \right) \Big\} \le e^{-t^2 /(2 L^2)}.
\end{equation}

For any $p \in [1, \infty]$, the conjugate $p'$ is defined by  $1/p + 1/p' = 1$ (with $1/\infty=0$).
We denote by $B_p^N$ the unit ball of the norm $\|\cdot\|_p$,  that is
\[
B_p^N = \left\{ x \in \RNN\, \,\, |\,\, \,   \|x\|_p = \left(\sum_{i=1}^{N} |x_i|^p \right)^{1/p} \le 1 \right\}
\]
(in the case $p=\infty$, $\|x\|_p=\|x\|_\infty=\max_{i\leq N} |x_i|$).
It is classical that $(B_p^N)^\circ = B_{p'}^N$.

Below we  use  $c, C, c_1, C_1, \ldots$ for universal constants that may differ
from line to line. If a constant depends on a particular parameter, this parameter
will be written as a subscript.

\section{The case of $p$-ellipsoids}
\label{sect-p-ell}

In this section we study  $p$-ellipsoids. Their random radii were studied in \cite{HKNPU, HPS}, where
several cases were left  open.
Let $1\leq p\leq \infty$. For a non-increasing non-negative sequence  $\sigma=\{\sigma_j\}_{j=1}^N$, we define
its associated $p$-ellipsoid  by
\begin{equation}
	\label{def:Ep}
	{\cal E}_{p, \sigma}= {\cal E}_p := \left\{ x \in \RNN\, \,\, \Big|\,\, \,  \sum_{i=1}^{N} \left(\frac{|x_i|}{\sigma_i}\right)^{p} \le 1 \right\}
\end{equation}
with corresponding adjustments for the case $p=\infty$, in which case we ask that $\max_{i\leq N} \frac{|x_i|}{\sigma_i}\leq 1$.
It is not difficult to check that
\begin{equation}
	\label{p-dual}
	\left({\cal E}_p\right)^\circ =
	\left\{ x \in \RNN\, \,\, \Big|\,\, \,  \sum_{i=1}^{N} \left(\sigma_i |x_i|\right)^{p'} \le 1 \right\}.
\end{equation}

%

\subsection{Geometric properties of $p$-ellipsoids}
\label{intro-p-ell}

We summarize the geometric properties of $p$-ellipsoids in the following lemma.
\begin{lemma}\label{geom-p-ell}
	Let $1\leq p \le\infty$ and $\sigma=\{\sigma_j\}_{j=1}^N$ be a sequence of non-negative numbers.
	Let $I\subset \{1, \ldots N\}$ be non-empty and $P_I$ denote the coordinate projection on $\RR^I$.
	Then
	\begin{equation*}
		 R\left( P_I {\cal E}_p\right) =
		\left\{
		\begin{array}{ll}
			\displaystyle \max_{i \in I} \sigma_i \qquad &  \textrm{when} \ p \le 2
			\\
			\displaystyle \left(\sum_{i\in I} \, \sigma_i^{\frac{2p}{p-2}} \right)^{1/2 - 1/p} \qquad &  \textrm{when} \ 2<p <\infty
			\\
			\displaystyle \left(\sum_{i\in I} \, \sigma_i^2 \right)^{1/2} \qquad &  \textrm{when} \ p =\infty
		\end{array}
		\right.\  .
	\end{equation*}
	Moreover for $p<\infty$,
	\begin{equation*}
		\sqrt{\frac{2}{\pi}}
		\left(\sum_{i\in I}\frac{1}{\sigma_i^{p}}\right)^{\frac{1}{p}}\leq
		\ell(P_I {\cal E}_p)\leq \max\left\{1, \sqrt{p-1}\right\}
		\left(\sum_{i\in I}\frac{1}{\sigma_i^{p}}\right)^{\frac{1}{p}}
	\end{equation*} 	
		and, in the case $p=\infty$,   writing $I=\{i_1, \ldots, i_\ell\}$ with $\sigma_{i_1} \le \ldots \le \sigma_{i_\ell}$,
	\[
	\ell(P_I {\cal E}_\infty)
	\approx  \max_{j\leq \ell } \frac{\sqrt{\log(2 j)}}{\sigma_{i_j}},
	\]
	where $\approx$ means inequalities in both directions up to absolute positive constants.
\end{lemma}

\medskip
\noindent
{\bf Proof.}
First we verify the claims for the radii. The case $p\leq 2$ is trivial (note that in this case ${\cal E}_p\subset {\cal E}_2$).
Let $2<p<\infty$. By H\"older's inequality, applied with $q=p/2$ and
$q'=p/(p-2)$ to a vector $x=\{x_i\}_{i=1}^{N}$,
\begin{equation}\label{Hold-diam}
	\left(\sum_{i\in I} x_i^2 \right)^{1/2}=
    \left(\sum_{i\in I} \sigma_i^2\,\left( \frac{|x_i|}{\sigma_i}\right)^{2}  \right)^{1/2}
	\le
	\left(\sum_{i\in I} \sigma_i^{\frac{2p}{p-2}} \right)^{1/2 - 1/p}
	\left(\sum_{i\in I} \left(\frac{|x_i|}{\sigma_i}\right)^{p}  \right)^{1/p}.
\end{equation}
As the above inequality attains equality at $x$  with $x_i=\sigma_i^{\frac{p}{p-2}}$, we get the equality for the radius.
The case $p=\infty$ is similar.

Next we evaluate the $\ell$-functionals. Assume first that  $p < \infty$. It is known that $c_p:=(\E\, |g|^{p})^{\frac{1}{p}}\leq \max\left\{1, \sqrt{p-1}\right\}$,
therefore
\[
\ell(P_I {\cal E}_p) = \E \left(\sum_{i\in I}  \frac{|g_i|^{p}}{\sigma_i^{p}} \right)^{\frac{1}{p}}
\leq\left( \E \sum_{i\in I}  \frac{|g_i|^{p}}{\sigma_i^{p}} \right)^{\frac{1}{p}}
=(\E\, |g|^{p})^{\frac{1}{p}}  \left(\sum_{i\in I}  \frac{1}{\sigma_i^{p}} \right)^{\frac{1}{p}}
\le c_p
\left(\sum_{i\in I}\frac{1}{\sigma_i^{p}}\right)^{\frac{1}{p}}.
\]
To estimate from below we consider the sequence $w = \{ |g_j| /\sigma_j \}_{j\in I}$
and write
\[
\E \left(\sum_{i\in I}  \frac{|g_i|^{p}}{\sigma_i^{p}} \right)^{\frac{1}{p}}
 = \E \|w\|_{p} \ge \| \E \, w\|_{p}  =\left(  \sum_{i\in I}  \left(\frac{\E |g_i|}{\sigma_i}\right)^p\right)^{\frac{1}{p}} =
\E|g| \left(\sum_{i\in I}\frac{1}{\sigma_i^{p}}\right)^{\frac{1}{p}}= \sqrt{\frac{2}{\pi}} \left(\sum_{i\in I}\frac{1}{\sigma_i^{p}}\right)^{\frac{1}{p}}.
\]

 Now, in the case $p=\infty$, $\ell(P_I {\cal E}_\infty) = \E \sup_{i\in I} \frac{|g_i|}{\sigma_i }$,
 and the claimed equivalence is known.
\qed

\subsection{Random radius of $p$-ellipsoids}
\label{cor-rand-rad}

In this section we apply our main theorems to obtain bounds on random radii of
 $p$-ellipsoids defined in \eqref{def:Ep}. We claim and discuss estimates which will be proved in
 Section~\ref{sec:ellipsoid}.

\smallskip

We start with the upper bound, which follows from  Theorem~\ref{MainThm}.

\begin{corollary}
	\label{cor:lp-ellipsoid}
There exist universal constants $C, C_1\ge 2$, and $0< \gamma  < 1$ such that  the following holds.
Let $1 \le k \le n < N$ and $0 < \eps \le 1/(2C)$.   Then
with probability at least
\[
1 - 3 (C \eps)^{n-k+1} - k e^{- \gamma n}
\]
we have the following.

\smallskip

\noindent
(1) If $2<p <\infty$, then
\[
{\cal R}_n({\cal E}_p ) \le
\frac{8 \sqrt{\log(1/\eps)}}{\eps} \,\, \sqrt{\frac{2n}{n-k +1}} \, \,
\left(\sum_{i>k} \sigma_i^{\frac{2p}{p-2}}\right)^\frac{p-2}{2p}
+
\frac{2 \sqrt n}{\eps (n-k+1)} \ \left(\sum_{i>n} \sigma_i^{p'}\right)^{1/p'}.
\]

\noindent
(2) If $1 < p \le 2$, then
\[
{\cal R}_n({\cal E}_p) \le
\frac{8 \sqrt{\log(1/\eps)}}{\eps} \,\, \sqrt{\frac{2n}{n-k +1}}\,  \, \sigma_{k+1}
+
\frac{2 \sqrt{p' \, n}}{\eps (n-k+1)} \ \left(\sum_{i>k} \sigma_i^{p'}\right)^{1/p'} .
\]

\noindent
(3) If $p = 1$, then
\[
{\cal R}_n({\cal E}_1) \le
\frac{C_1 \sqrt{\log(1/\eps)}}{\eps}\,\,  \sqrt{\frac{ n}{n-k +1}}  \,\, \sigma_{k+1}
+
\frac{2 \sqrt{n}}{\eps (n-k+1)} \ \sup_{i>k} \sigma_i \sqrt{\log(2 (i-k))}\ .
\]
\end{corollary}

\begin{remark}
Note that in the case $1  < p \le 2$,
$$
  \left(\sum_{i>k} \sigma_i^{p'}\right)^{1/p'} \leq \left(\sum_{i>n} \sigma_i^{p'}\right)^{1/p'}
  + (n-k)^{1/p'}  \sigma_{k+1}  \leq \left(\sum_{i>n} \sigma_i^{p'}\right)^{1/p'}
  + (n-k+1)^{1/2}  \sigma_{k+1},
$$
therefore we can have the sum in the upper bound start from $i=n+1$ (as in the case $p>2$) at the cost of multiplying by a factor of $\sqrt{p'}$ the first term of the upper bound too.
\end{remark}

Next we provide two lower bounds, which are consequences of Theorems~\ref{LowerThm} and \ref{LowerGeneralThm}.

\begin{corollary}
	\label{cor:lp-ellipsoid-lower}
Let $1 \le n < N$ and $1\leq p<\infty$.
For every $t>0$, with probability at least
\[
1 - e^{-n/8} - e^{-t^2/2}
\]
the following holds.
If $2\le p <\infty$,
\[
 {\cal R}_n({\cal E}_p ) \ge
 \frac{\sqrt n}{2 \left( \displaystyle \sqrt{p} \bigg( \sum_{i=1}^{n+1} \frac{1}{\sigma_i^p}\bigg)^{1/p} + \frac{t}{\sigma_{n+1}} \right)}\ .
\]
If $1 \le p < 2$,
\[
{\cal R}_n({\cal E}_p ) \ge
\frac{\sqrt n}{2 \left( \displaystyle  \bigg( \sum_{i=1}^{n+1} \frac{1}{\sigma_i^p}\bigg)^{1/p} + t \bigg( \sum_{i=1}^{n+1} \frac{1}{\sigma_i^{2p/(2-p)}}\bigg)^{1/p - 1/2} \right)}\ .
\]
\end{corollary}

\begin{corollary}
	\label{LowerEllpThm}
Let $1 \le n < N$. For  $k<N$ and $1\leq p\leq \infty$ set
$$
  r_{k, p} =
  \left\{
	\begin{array}{ll} \sigma_1 \qquad &  \textrm{when} \ 1\le  p \le 2
		\\
		\displaystyle \left(\sum_{i=1}^k \, \sigma_i^{\frac{2p}{p-2}} \right)^{1/2 - 1/p} \qquad &  \textrm{when} \ 2<p <\infty
        \\
        \displaystyle \left(\sum_{i=1}^k \, \sigma_i^2 \right)^{1/2} \qquad &  \textrm{when} \ p =\infty
	\end{array}
	\right.\ .
$$
For $p>1$ denote by $I_p$ the set of all $k<N$ satisfying
\begin{equation}
	\label{eq:hypLower}
	\bigg( \sum_{j>k} \sigma_j^{p'} \bigg)^{1/p'}
	\ge
	\left\{
	\begin{array}{ll}
       8 \sqrt{ n} \,  \sigma_{k+1} \qquad &  \textrm{when} \ 1 < p \le 2
 \\
       \displaystyle  8 \sqrt{ n} \bigg(\sum_{j>k} \sigma_j^{\frac{2p}{p-2}} \bigg)^{1/2 - 1/p} \qquad &  \textrm{when} \ 2<p\leq \infty
	\end{array}
	\right.
\end{equation}
(here, for $p=\infty$,  $2p/(p-2)$ is understood as $(1/2-1/p)^{-1}=2$).  Assume $I_p\neq \emptyset$.
Then
with probability at least  $1 - 2e^{- 2 n}$ one has
\[
{\cal R}_n({\cal E}_p )
\ge  \max_{k\in I_p}
\Bigg(
 \frac{1}{r_{k, p}} +
 8 \, \sqrt{n} \, \bigg( \displaystyle \sum_{j>k} \sigma_j^{p'} \bigg)^{-1/p'}
\Bigg)^{-1}
\geq
\frac{1}{2} \, \max_{k\in I_p}\, \min\Bigg\{ r_{k, p}, \,
\frac{1}{8 \, \sqrt{n}} \bigg( \sum_{j>k} \sigma_j^{p'} \bigg)^{1/p'}
\Bigg\}.
\]
Moreover, there exist absolute positive constants $c$ and  $C$ such that,
if we set $$\eta_k:=\max_{i>k } \sigma_{i}\sqrt{\log(2 (i-k))},$$ then for $p=1$ we will have,
with probability at least  $1 - 2e^{- 2 n}$,
\[
{\cal R}_n({\cal E}_p )
\ge  \max_k
\left(
 \frac{1}{\sigma_1} +
\frac{ \sqrt{n}}{ c \,\eta_k}
\right)^{-1}\geq \frac{1}{2} \max_k \min\left\{ \sigma_1, \, \frac{ c\, \eta_k}{ \sqrt{n}} \right\},
\]
where the maximum is taken over all $1\leq k<N$ satisfying
 $
    \eta_k\geq C \sqrt{n} \, \sigma_{k+1} .
 $
\end{corollary}

\begin{remark}
	We did not try to optimize the constant $8$ appearing in the above proposition. Careful analysis
of our proofs shows that, given $p>1$ and any $\eps, \delta \in (0,1)$, the constant $8$ in
condition (\ref{eq:hypLower}) can be substituted by $\sqrt{\pi/2}\, (1+\eps)/(1-\delta)$, while the constant $8$ appearing
in the lower bound for ${\cal R}_n$  can be substituted by $\sqrt{\pi/2}\, (1+\eps)/\delta$, with the claimed
probability being at least $1 - 2\exp{(- \eps^2n/2)}$.
\end{remark}

\begin{remark}
  Assume for instance that $p=2$ and $\sum_{j>k} \sigma_j^{2} \geq 64 n\,\sigma_k^2
  \geq 64 n\,\sigma_{k+1}^2$. Then
  Corollary~\ref{LowerEllpThm} can be applied, and it implies that
  \[
{\cal R}_n({\cal E}_2 )
\ge
\left(
 \frac{1}{\sigma_1} +
8 \, \sqrt{n} \, \bigg( \displaystyle \sum_{j>k} \sigma_j^{2} \bigg)^{-1/2}
\right)^{-1}\geq \frac{ \sigma_{k}}{2},
\]
  which in particular recovers Theorem~5 of \cite{HKNPU}, where
  the condition $\sum_{j>k} \sigma_j^{2} \geq Cn \,\sigma_{k}^2$
  was used.
\end{remark}

\begin{remark}\label{ill-ell}
  As we mentioned in the Introduction, our two lower bounds are of different nature. Corollary~\ref{cor:lp-ellipsoid-lower} provides estimates in terms of the largest axes of the ellipsoid (what we called ``large" part of the body in Remark~\ref{lower-nat}), while
Corollary~\ref{LowerEllpThm} works with the smallest axes of the ellipsoid (what we called ``small" part of the body).
\end{remark}

\section{Proofs of main theorems}
\label{sec:upper}

\subsection{Proof of Theorem~\ref{MainThm}}
Fix $k \ge 1$ and  a subspace $E$ of codimension $k$.  By the rotational invariance of the Haar measure on the Grassmannian
and of the Gaussian measure,
we can assume that $E = (\Rk)^\perp $ and $E^\perp = \Rk$. Let $y$ be a point in $K \cap \ker \Gn $, then
\begin{equation}
	\label{eq:decompose}
y = P_E y + P_{E^\perp} y .
\end{equation}
Since $y \in K$,  we have by definition,
\begin{equation}
	\label{eq:first}
|P_E y | \le  \rad (P_E K).
\end{equation}
Our goal is to bound $|P_{E^\perp} y| = |P_k y|$. Recall that
$\Gn : \RNN \to \Rn$ is a random matrix with i.i.d. ${\cal N}(0,1)$ Gaussian entries and that
$\Gkn : \Rk \to \Rn$ is its restriction to $\Rk = E^\perp$. Define by $A$ the left inverse of $\Gkn$ that is
\[
A = ( \Gkn^* \Gkn)^{-1} \Gkn^*,
\]
so that $A \Gkn = \Id$. Observe that $A$ is independent of ${\Gamma_n^k} = \Gn - \Gkn$, and
since $y \in \ker \Gn$,
\[
P_k y = A \Gkn P_k y = A \Gn P_k y = A \Gn (P_k y - y) = A {\Gamma_n^k} P_E y.
\]
Since $A$ and ${\Gamma_n^k}$ are independent, we start by bounding
\[
b_0 = \max_{y \in K\cap \ker \Gn} |P_k y| = \max_{z \in P_E K} |A {\Gamma_n^k}  z|
\]
conditionally on $A$.
We have
\begin{align*}
	b_0 & = \max_{z \in P_E K} \max_{v \in B_2^k} \langle A {\Gamma_n^k} z, v \rangle = \max_{z \in P_E K} \max_{v \in B_2^k} \langle {\Gamma_n^k} z, A^* v \rangle
	\\
	& = \max_{z \in P_E K} \| {\Gamma_n^k} z\|_{(A^* B_2^k)^\circ} = \| {\Gamma_n^k} : P_E K \to (A^* B_2^k)^\circ \|.
\end{align*}
From the Chevet-Gordon inequality \cite{Chevet1977, Gordon1985} (note that it also holds in the non-symmetric case with the same proof), we obtain that
\[
\E b_0 \le \max_{z \in P_E K} |z| \ \ell((A^* B_2^k)^\circ) + \max_{w \in A^* B_2^k} |w| \ \ell((P_E K )^\circ).
\]
Let $G=G^{(n)}$ (that is, $G$ is a standard Gaussian vector in $\Rn$).
By definition we have
\[
\ell((A^* B_2^k)^\circ) = \E \max_{w \in A^* B_2^k} \langle G, w \rangle = \E \max_{v \in B_2^k} \langle AG, v \rangle = \E |AG| \le \|A\|_{HS}
\]
and
\[
\max_{w \in A^* B_2^k} |w| = \max_{v \in B_2^k} |A^* v| = \|A^*\| = \|A\|.
\]
Therefore we conclude that, conditionally on $A$,
\[
\E b_0 \le \rad(P_E K) \ \|A\|_{HS} + \ell((P_E K )^\circ) \ \|A\|.
\]
Moreover, since the map $T \mapsto \| T : P_E K \to (A^* B_2^k)^\circ \|$ is $L$-Lipschitz with respect
to the Hilbert-Schmidt norm with $ L \le \|A\| \rad(P_EK)$,
by the Gaussian concentration inequality \eqref{ineq:concentration} we get that, for every $u > 0$,
\[
\P \left(b_0 \ge u \|A\| \rad(P_E K) + \E  \| \Gn : P_E K \to (A^* B_2^k)^\circ \| \right)
\le
e^{-u^2/2}.
\]
We choose
\[
u = \sqrt{\log(1/\eps)} \sqrt{2(n-k+1)}.
\]
Then with probability at least
$1- \eps^{n-k+1}$,
\begin{equation}
	\label{eq:mainupper}
b_0 \le \sqrt{\log(1/\eps)} \sqrt{2(n-k+1)} \|A\| \rad(P_E K) + \rad(P_E K) \ \|A\|_{HS} + \ell((P_E K )^\circ) \ \|A\|.
\end{equation}
Since $A$ is the left-inverse of $\Gkn$ and for every $i \le k$, $s_i(\Gkn) \ge s_i(\Gamma_{in})$, we observe that
\[
\|A\| = \frac{1}{s_k(\Gkn)} \qquad \textrm{and} \qquad \|A\|_{HS} = \sqrt{\sum_{i=1}^{k} \frac{1}{s_i^2(\Gkn)}} \le \sqrt{\sum_{i=1}^{k} \frac{1}{s_i^2(\Gamma_{in})}}.
\]
By Theorem 1.1 in \cite{RV2008}, for some absolute constant $\gamma>0$ with probability at least
\[
1 - (C \eps)^{n-i+1} - e^{- \gamma n}
\]
we have $s_i(\Gamma_{in}) \ge \eps (\sqrt n - \sqrt{i-1})$. Since $\eps \le 1/(2C)$,
\[
\sum_{i=1}^{k}(C \eps)^{-i} \le (C \eps)^{-k-1} / ((C \eps)^{-1} - 1)
\le 2 (C \eps)^{-k}.
\]
Thus with probability  at least
\begin{equation}
	\label{eq:proba}
1 - 2  (C \eps)^{n-k+1} - k e^{- \gamma n}
\end{equation}
we obtain that for every $i \in \{1, \ldots, k \}$,
$s_i(\Gamma_{in}) \ge \eps (\sqrt n - \sqrt{i-1}) \ge \eps \frac{n-i+1}{2 \sqrt n}$.
Therefore,
\[
\|A\| \le  \frac{2 \sqrt n}{\eps (n-k+1)}
\quad  \quad\textrm{and} \quad \quad
\|A\|_{HS} \le \frac{2}{\eps} \sqrt{\sum_{i=1}^{k} \frac{n}{(n-i+1)^2}}
\le \frac{2}{\eps} \sqrt{\frac{2n}{n-k+1}}.
\]
Combining this with \eqref{eq:mainupper} and \eqref{eq:proba}, we conclude that with probability at least
\[
1 - \eps^{n-k+1} - 2  (C \eps)^{n-k+1} - k e^{- \gamma n}
\]
one has
\[
b_0 \le \frac{2}{\eps} \left( (1 +\sqrt{\log (1/\eps)}) \rad(P_E K) \ \sqrt{\frac{2n}{n-k+1}} + \ell((P_E K )^\circ) \  \frac{\sqrt n}{n-k+1} \right).
\]
Since $\eps \le 1/(2C)$, we can assume that $1 \le \log(1/\eps)$; we also observe that for $1 \le k \le n$, we have
$\sqrt{\frac{n}{n-k+1}}\ge 1$. By  \eqref{eq:decompose} and \eqref{eq:first},  $|y| \le \rad(P_E K) + b_0$ and we finally obtain that with probability at least
\[
1 - 3 (C \eps)^{n-k+1} - k e^{- \gamma n}
\]
one has
\[
\forall y \in \ker \Gn \cap K, \quad |y| \le \left(  \frac{6 \sqrt{\log(1/\eps)}}{\eps} \sqrt{\frac{2n}{n-k+1}}\right)  \rad(P_E K) +
 \frac{2 \sqrt n}{\eps (n-k+1)} \ \ell((P_E K )^\circ)
\]
as required.
\qed

\begin{remark}
	We could also have adopted an algorithmic point of view as in \cite{HKNPU}. Let $y \in K$.
Then by the decomposition $y = P_k y + P_E y$
	and $\Gn = \Gkn + \Gamma_n^k$, one has
	\[
	| y - A \Gn y | = | P_E y - A \Gamma_n^k P_E y | \le  |P_E y | + | A \Gamma_n^k P_E y | \le R(P_E K) + b_0.
	\]
	This shows that the proof not only gives an upper bound for the Euclidean norm of any point $y \in K \cap \ker \Gn$ but also of an approximation
	of any $y \in K$ by its reconstruction via the (random) linear algorithm defined by $A$.
\end{remark}

\subsection{Proof of Theorem \ref{LowerThm}}
\label{sec:lower}

Our strategy to obtain a lower bound for ${\cal R}_n(K)$ is based on the fact that
a random subspace of codimension $n$ in $\RR^N$ is distributed as the image of a
Gaussian operator acting from $\RR^{N-n}$ to $\RNN$.

Let $F$ be an $(n+1)$-dimensional subspace of $\RNN$.  By rotational invariance of the Gaussian measure, we can assume that $F = \RN$.
Let $G$ be a Gaussian matrix $G : \RR^{N-n} \to \RR^N$ with independent standard Gaussian entries. We see this matrix as the matrix
with $N$ rows $G_1, \ldots G_N$, where $G_i$'s are independent Gaussian ${\cal N}(0, \Id)$ random vectors in $\RR^{N-n}$.
Then $\im G$ is uniformly distributed as
a random subspace of codimension $n$ in $\RR^N$. Therefore, by definition \eqref{eq:def},  ${\cal R}_n(K)$ may be written as
\[
{\cal R}_n(K) = \sup_{y \in \RR^{N-n}\setminus\{0\}} \frac{|Gy|}{\|Gy\|_K}.
\]
Observe that $Gy$ is the vector with coordinates $\langle G_i, y \rangle$,  $1\leq i\leq N$.
Define $E$ as the random subspace spanned by $\{G_{n+2}, \ldots, G_N\}$ and set $\overline{y} = P_{E^\perp} G_1$, so that
\begin{align*}
G\overline{y}  & = (\langle G_1, P_{E^\perp} G_1 \rangle, \ldots,\langle G_{n+1}, P_{E^\perp} G_1 \rangle, 0, \ldots, 0)
\\
& = (\langle P_{E^\perp} G_1, P_{E^\perp} G_1 \rangle, \ldots,\langle P_{E^\perp} G_{n+1}, P_{E^\perp} G_1 \rangle, 0, \ldots, 0).
\end{align*}
Almost surely, $E$ has dimension $N-n -1$, hence $E^\perp$ has dimension 1.
Moreover, it is independent of $G_1, \ldots, G_{n+1}$, hence, by properties of Gaussian vectors, $G\overline{y} $ is distributed as
\[
(h_1^2, h_1 h_2, \ldots, h_1 h_{n+1}) = h_1 (h_1, \ldots, h_{n+1})
\]
in $\RN$,  where $h_1, \ldots, h_{n+1}$ are independent standard Gaussian random variables. Therefore, almost surely,
\[
{\cal R}_n(K) \ge \frac{|G\overline{y}|}{\|G\overline{y}\|_K} \sim \frac{|G_F|}{\|G_F\|_K},
\]
where $G_F$ is a standard $(n+1)$-dimensional Gaussian random vector in $F = \RR^{n+1}$.
It is well known that $\E |G_F| \ge \sqrt n$ (to show this, one can use, say, the Poincar\'e inequality; see e.g. \cite[Chapter~3.1]{Ledoux}).
Applying \eqref{ineq:concentration} with $t = \sqrt{n}/2$, one has
\[
\P (|G_F| \le \sqrt{n} / 2) \le  e^{-t^2/2} = e^{-n/8}.
\]
At the same time, since $z \mapsto \|z\|_{K \cap F}$ is $\rad((K \cap F)^\circ)$-Lipschitz with respect
to the Euclidean norm, \eqref{ineq:concentration} implies that
\[
\forall t > 0,
\quad
\P \left( \|G_F\|_{K\cap F} \ge \E \|G_F\|_{K\cap F} + t \ \rad((K \cap F)^\circ) \right) \le e^{-t^2/2}.
\]
Combining both bounds, we obtain that for every $t>0$,
\[
\P \left( \frac{|G_F|}{\|G_F\|_K} \le \frac{\sqrt{n}}{2 (\ell({K\cap F}) + t \ \rad((K \cap F)^\circ))} \right)
\le
e^{-n/8} + e^{-t^2/2}.
\]
This implies the desired result.
\qed

\subsection{Proof of Theorem~\ref{LowerGeneralThm}}

We will use the  following statement, which  is a classical tool from Asymptotic Geometric Analysis.
Given that it is not usually stated in this form in various textbooks,  for the sake of completeness here,
we include it here in full generality and provide the proof.
In particular, special attention is paid to the values of the constants, the fact that the convex body may not be symmetric with respect to the origin, and to the fact that the result is valid for a subset of the unit sphere.

\begin{theorem}
	\label{thm:Gordon}
Let $L \subset \RR^m$ be a convex body containing the origin in its interior
and $T$ be a closed subset of $S^{n-1}$.
Let $\Gamma : \ell_2^n \to (\RR^m, \| \cdot \|_L)$ be a Gaussian operator given as a matrix with i.i.d.
standard Gaussian entries. Then with probability at least $1 - \exp(-u^2 /2)$, one has
\[
	\inf_{x \in T} \| \Gamma x \|_L \ge \ell(L) - R(L^\circ)( \ell_*(T)+u),
\]
and with probability  at least  $1 - \exp(-u^2 /2)$, one has
\[
	\sup_{x \in T} \| \Gamma x \|_L \le \ell(L) + R(L^\circ)( \ell_*(T)+u).
\]
\end{theorem}

\medskip

\noindent
{\bf Proof.} Applying   Theorem~2.1 from \cite{Gordon1985}
with  $\mathcal{E}=T$ and $\Theta= L^\circ$, similarly to the proofs of
Theorem~2.5 in \cite{Gordon1985} or Corollary~1.2 in \cite{Gordon1988}, we obtain
\[
\E \inf_{x \in T} \| \Gamma x \|_L \ge \ell(L) - R(L^\circ) \ell_*(-T)
\]
and
\[
\E \sup_{x \in T} \| \Gamma x \|_L \le \ell(L) + R(L^\circ) \ell_*(T).
\]
By symmetry of the Gaussian random variables, $\ell_*(-T)$ can be replaced by $\ell_*(T)$ in the lower bound.
Let $F_{min}$ and $F_{max}$ be the functions defined on the space of $m \times n$ matrices by
\[
F_{min}(A) = \inf_{x \in T} \| A x \|_L
\qquad
\mathrm{and}
\qquad
F_{max}(A) = \sup_{x \in T} \| A x \|_L.
\]
Then $F_{min}$ and $F_{max}$ are $R(L^\circ)$-Lipschitz with respect to the Hilbert-Schmidt norm. Therefore, a direct application of the Gaussian concentration inequality
\eqref{ineq:concentration} proves the result.
\qed

\medskip

\noindent
{\bf Proof of Theorem \ref{LowerGeneralThm}.}
By the rotational invariance of the Gaussian measure, we can assume that $E^\perp = \RR^k$
(considered as the subspace of vectors in $\RNN$ having zero coordinates starting from the $(k+1)$-th one).
Let $E$ be equipped with the norm whose unit ball is $L = (P_E K)^\circ$.
Recall that $\Gamma_n$, $\Gamma_n^k$, and $(\Gamma_n^k)^T : \ell_2^n \to E$ were introduced
in Section~\ref{sec:notations}, and note that  $\Gamma_n^k$ and $\Gamma_{kn}$ are independent.
We start by working conditionally on $\Gamma_{kn}$.
An application of Theorem~\ref{thm:Gordon} with $T = S^{n-1}$
(so that $\ell_*(T) =\ell(B_2^n)$)
and $u = 2 \sqrt{n}$ shows that, with probability at least $1 - \exp(-2n)$, one has
\[
\inf_{x\in S^{n-1}} \| (\Gamma_n^k)^T  \, x \|_{(P_E K)^\circ} \ge \ell_*(P_E K) - 3 \sqrt{n}\, R(P_E K) \ge  \ell_*(P_E K) / 2,
\]
where we used  $\ell((P_E K)^\circ)=\ell_*(P_E K)$, $\ell(B_2^n) \le \sqrt{n}$, and our assumption \eqref{assumption:key}.
Therefore, for every $x\in \RR^n$,
$$
 \frac{\ell_*(P_E K) }{2}  |x| \leq \| (\Gamma_n^k)^T  \, x \|_{(P_E K)^\circ}
  =\max_{y\in P_E K} \langle (\Gamma_n^k)^T  \, x, y\rangle =
  \max_{y\in P_E K} \langle  x, \Gamma_n^k y\rangle =
  \max_{z\in \Gamma_n^k P_E K} \langle  x, z\rangle,
$$
which gives
\begin{equation}
	\label{eq2:keyinclusion}
\frac{\ell_*(P_E K) }{2} B_2^n \subset \Gamma_n^k (P_E K).
\end{equation}
Let $\theta_k \in E^\perp\cap S^{N-1}$ be a vector with the property that
\[
\|-\theta_k\|_K = \inf_{x\in E^\perp\cap S^{N-1}} \|x\|_K = \frac{1}{R(K \cap E^\perp)}.
\]
By \eqref{eq2:keyinclusion},
there exists a vector $u \in P_E K$ such that
\[
\frac{\ell_*(P_E K)}{2 |\Gn (\theta_k)|} \Gn(\theta_k) =  \Gamma_n^k(u)
\]
(note that $\Gn (\theta_k) \ne 0$ with probability 1).
Since $u \in E$, we have $ \Gamma_n^k(u)= \Gn u$.
 Therefore,
\[
z := \frac{2 |\Gkn(\theta_k)|}{\ell_*(P_E K)}\, u - \theta_k \in \ker \Gn.
\]
Since $u \in P_E K\subset E$ and $\theta_k \in E^\perp$ have disjoint supports, we have $|z| \ge |\theta_k| = 1$.
Moreover, by the definition of $\delta_{E,K}$, we have
 $\|u\|_K \le \delta_{E,K} \|u\|_{P_E K} \le \delta_{E,K}$. Hence
\[
\|z\|_K \le  \frac{2 | \Gkn(\theta_k)|}{\ell_*(P_E K)} \|u\|_K +  \|-\theta_k\|_K
\le
\frac{ 2 |\Gkn(\theta_k)| \delta_{E,K}}{\ell_*(P_E K)}  + \frac{1}{R(K \cap E^\perp)}.
\]
Therefore, conditionally on $\Gkn$, with probability greater than $1 - e^{- 2 n}$ we have
\[
\frac{|z|}{\|z\|_{K}}
\ge
\left( \frac{1}{R(K \cap E^\perp)} +  \frac{2 \delta_{E,K}|\Gkn(\theta_k)|}{\ell_*(P_E K)}\right)^{-1}.
\]
Since $\theta_k$ is a fixed vector from the unit sphere, $\Gkn(\theta_k)$ is distributed as a standard Gaussian vector in $\RR^n$.
The concentration inequality \eqref{ineq:concentration} implies that $|\Gkn(\theta_k)| \le 3\sqrt{n}$ with probability at least
 $1 - e^{-2n}$. Therefore, with probability greater than $1- 2 e^{-2 n}$, we have
\[
    {\cal R}_n(K) \ge \frac{|z|}{\|z\|_{K}}
    \ge \left( \frac{1}{R(K \cap E^\perp)} +  \frac{6 \, \delta_{E,K} \, \sqrt{n}}{\ell_*(P_E K)} \right)^{-1}
    \ge \frac{1}{2} \, \min \left\{ R(K \cap E^\perp),\,  \frac{\ell_*(P_E K)}{6 \, \delta_{E,K} \, \sqrt{n}} \right\}.
\]
This  completes the proof.
\qed

\section{Deriving bounds for $p$-ellipsoids}
\label{sec:ellipsoid}

In this section, we provide proofs of Corollaries~\ref{cor:lp-ellipsoid},
\ref{cor:lp-ellipsoid-lower}  and  \ref{LowerEllpThm}.

\subsection{Proof of Corollary~\ref{cor:lp-ellipsoid}}
	
For any $1 \le k \le n$, we choose $E = (\Rk)^\perp$, so that
\[
P_E {\cal E}_p = \left\{ x \in E \,\, \Big|\,\,  \sum_{i > k} \left(\frac{|x_i|}{\sigma_i}\right)^{p} \le 1 \right\}
\quad \quad
\mathrm{and}
\quad\quad
\forall y \in E\,\,\,\,\,\,
\|y\|_{(P_E {\cal E}_p )^\circ} = \left(\sum_{i>k} (\sigma_i |y_i|)^{p'} \right)^{1/p'}.
\]	
Assume $p>2$.  We  evaluate the parameters $\rad(P_E {\cal E}_p)$ and $\ell_*(P_E {\cal E}_p)$ appearing in the statement of Theorem~\ref{MainThm}. By Lemma~\ref{geom-p-ell} and (\ref{p-dual}),
\[
	\rad(P_E {\cal E}_p) \le \left(\sum_{i> k} \sigma_i^{\frac{2p}{p-2}} \right)^{1/2 - 1/p}
\]
and, as $p'<2$,
\[
\ell_*(P_E {\cal E}_p)= \ell((P_E {\cal E}_p)^\circ) = \ell( P_E {\cal E}_{p', 1/\sigma})
\le
\left(\sum_{i>k} \sigma_i^{p'} \right)^{1/p'}.
\]
Moreover,  since
$$1/p' = 1 - 1/p > 1/2 - 1/p,$$  H\"older's inequality yields
\begin{align*}
	\left(\sum_{i>k} \sigma_i^{p'} \right)^{1/p'}
	& \le
	\left(\sum_{i=k+1}^n \sigma_i^{p'} \right)^{1/p'}
	+
	\left(\sum_{i>n} \sigma_i^{p'} \right)^{1/p'}
	\\
	& \le
	\sqrt{n-k} \left(\sum_{i= k+1}^n \sigma_i^{\frac{2p}{p-2}} \right)^{1/2 - 1/p}
	+
	\left(\sum_{i>n} \sigma_i^{p'} \right)^{1/p'}
	\\
	& \le
	\sqrt{n-k} \left(\sum_{i>k} \sigma_i^{\frac{2p}{p-2}} \right)^{1/2 - 1/p}
	+
	\left(\sum_{i>n} \sigma_i^{p'} \right)^{1/p'}.
\end{align*}
This implies
\[
\frac{\sqrt n}{n-k+1} \left(\sum_{i>k} \sigma_i^{p'} \right)^{1/p'}
\le
\sqrt{\frac{n}{n-k+1}} \left(\sum_{i>k} \sigma_i^{\frac{2p}{p-2}} \right)^{1/2 - 1/p}
+
\frac{\sqrt n}{n-k+1} \left(\sum_{i>n} \sigma_i^{p'} \right)^{1/p'}.
\]
Applying Theorem~\ref{MainThm}, we conclude that for every $1 \le k \le n$ with probability at least
\[
1 - 3 (C \eps)^{n-k+1} - k e^{- \gamma n}
\]
one has
\[
{\cal R}_n({\cal E}_p)
\le
\left(\frac{8 \sqrt{\log(1/\eps)}}{\eps}  \sqrt{\frac{2n}{n-k +1}}\right) \, \left(\sum_{i>k} \sigma_i^{\frac{2p}{p-2}}\right)^\frac{p-2}{2p}
+
\frac{2}{\eps} \frac{\sqrt{n}}{n-k+1} \left(\sum_{i>n} \sigma_i^{p'}\right)^{1/p'}.
\]
The case $1 \le p \le 2$ is treated along the same lines upon observing that by Lemma~\ref{geom-p-ell} and (\ref{p-dual}),
$\rad(P_E {\cal E}_p) = \sigma_{k+1}$  (for all $1 \le p \le 2$), and for
 $1  < p \le 2$,
$$
\ell((P_E {\cal E}_p)^\circ) \leq \sqrt{p'}
\left(\sum_{i>k} \sigma_i^{p'} \right)^{1/p'},
$$
while for  $p=1$,
$$\ell((P_E {\cal E}_p)^\circ) = \ell_*(P_E {\cal E}_p) \leq C \sup_{i>k} \sigma_i \sqrt{\log(2 (i-k))}.$$
This completes the proof.
\qed

\subsection{Proof of Corollary~\ref{cor:lp-ellipsoid-lower}}
We choose $F = \RR^{n+1}$. Then
\[
\forall x \in \RR^{n+1}\quad \quad
\|x\|_{{\cal E}_p\cap \,F} = \left(\sum_{i=1}^{n+1} \left(\frac{|x_i|}{\sigma_i}\right)^{p} \right)^{1/p}
\]
and
\[
({\cal E}_p\cap \,F)^\circ =
\left\{ x \in \RR^{n+1}\,\, \Big|\,\,  \sum_{i=1}^{n+1} (\sigma_i |x_i|)^{p'} \le 1 \right\}.
\]	

\noindent
{\bf Case 1. $2\le p <\infty$.}
 Then $1 < p' \le 2$ and by Lemma~\ref{geom-p-ell} together with (\ref{p-dual}),
$$
  \rad(({\cal E}_p \cap F)^\circ) =\rad(P_F({\cal E}_p)^\circ)=
  \rad(P_F {\cal E}_{p',1/\sigma})= 1/\sigma_{n+1}
$$
and
\[
\ell({\cal E}_p \cap F) = \ell(P_F {\cal E}_p )
\le \sqrt p
\left(\sum_{i=1}^{n+1} \frac{1}{\sigma_i^p} \right)^{1/p}.
\]
A straightforward application of Theorem \ref{LowerThm} gives the announced result.

\smallskip

\noindent
{\bf Case 2. $1 < p < 2$.}
Then $2 < p'<\infty$.
Lemma~\ref{geom-p-ell} together with (\ref{p-dual}) implies
\[
   \rad(({\cal E}_p \cap F)^\circ)=\rad(P_F {\cal E}_{p',1/\sigma})=
   \left(\sum_{i=1}^{n+1} \sigma_i^{-2p'/(p'-2)} \right)^{1/2-1/p'}
\]
and
\[
\ell({\cal E}_p \cap F) = \ell(P_F {\cal E}_p )
\le
\left(\sum_{i=1}^{n+1} \frac{1}{\sigma_i^p} \right)^{1/p}.
\]
Since $2p'/(p'-2) = 2p/(2-p)$, a straightforward application of Theorem \ref{LowerThm} completes this case.

\smallskip

\noindent
{\bf Case 3. $p=1$.}
This case is obtained similarly.
\qed

\subsection{Proof of Corollary~\ref{LowerEllpThm}}

To apply Theorem~\ref{LowerGeneralThm} we choose $E\subset \RNN$ to be the subspace of all vectors having first
$k$ coordinates equal to $0$, in particular $E^\perp=\RR^k$.  We clearly have
$P_{E^\perp} {\cal E}_p = {\cal E}_p\cap E^\perp$ and $P_E {\cal E}_p = {\cal E}_p\cap E$,
therefore $\delta_{E, K}=1$. By Lemma~\ref{geom-p-ell},
\begin{equation*}
     R\left( {\cal E}_p\cap E^\perp\right) =R\left( P_{E^\perp} {\cal E}_p\right)=
	\left\{
	\begin{array}{ll} \sigma_1 \qquad &  \textrm{when} \ 1\le  p \le 2
		\\
	\displaystyle	\left(\sum_{i=1}^k \, \sigma_i^{\frac{2p}{p-2}} \right)^{1/2 - 1/p} \qquad &  \textrm{when} \ 2<p <\infty
        \\
       \displaystyle \left(\sum_{i=1}^k \, \sigma_i^2 \right)^{1/2} \qquad &  \textrm{when} \ p =\infty
	\end{array}
	\right.\ ,
\end{equation*}
\begin{equation*}
    R\left( P_{E} {\cal E}_p\right)=
	\left\{
	\begin{array}{ll} \sigma_{k+1} \qquad &  \textrm{when} \ 1\le  p \le 2
		\\
	\displaystyle	\left(\sum_{i>k} \, \sigma_i^{\frac{2p}{p-2}} \right)^{1/2 - 1/p} \qquad &  \textrm{when} \ 2<p <\infty
        \\
      \displaystyle  \left(\sum_{i>k} \, \sigma_i^2 \right)^{1/2} \qquad &  \textrm{when} \ p =\infty
	\end{array}
	\right.
\end{equation*}
and
\begin{equation*}
\ell_*(P_E \, {\cal E}_p)=\ell(P_E \, {\cal E}_{p', 1/\sigma})\geq
\sqrt{\frac{2}{\pi}}
\left(\sum_{i>k}\sigma_i^{p'}\right)^{\frac{1}{p'}}
\end{equation*}
whenever $1<p\leq \infty$, while for $p=1$
$$
 \ell_*(P_E \, {\cal E}_\infty)=\ell(P_E \, {\cal E}_{p', 1/\sigma}) \geq c \,
 \max_{i>k } {\sigma_{i}} \sqrt{\log(2 (i-k))}\, ,
$$
where $c>0$ is an absolute constant.
Since $8\sqrt{\frac{2}{\pi}} > 6$, a direct application of Theorem~\ref{LowerGeneralThm} implies the result.
\qed

\section{Discussion and examples}
\label{sec:example}
In \cite{HKNPU}, the authors study the random radius of sections of an ellipsoid, which is the case $p=2$.
In particular they proved the following theorem.

\medskip
\noindent
\textbf{Theorem A.} \  \cite[Theorem~4]{HKNPU}
{\it For every $s,t  >1$ and  every $1 \le n < N$, one has
\begin{equation}
	\label{thmA}
{\cal R}_n({\cal E}_2)
\le
14 s n \left( \sum_{j >  n} \sigma_j^{2} \right)^{1/2}
\end{equation}
with probability at least  $1 - e^{- c^2 n}-c\sqrt{2e} /s$.
}

\medskip


By choosing $k=n$ in  Corollary~$\ref{cor:lp-ellipsoid}$, we obtain a better estimate
by a factor of $\sqrt{n}\,$:
$$
  {\cal R}_n({\cal E}_2) \le  C s \sqrt{n}  \left( \sum_{j >  n} \sigma_j^{2} \right)^{1/2}
$$
with probability at least  $1 - e^{- \gamma n} - (C\log s)/s$. We would like to note that
it does seem that the proof in \cite{HKNPU} can be adjusted to get
$$
 {\cal R}_n({\cal E}_2 ) \le   C\left( \frac{n}{n-k +1} \, \sigma_{k+1} +
   \frac{ \sqrt n}{n-k+1} \ \left(\sum_{i>k} \sigma_i^{2}\right)^{1/2}\right)
$$
(with a worse probability), which  is worse than the bound from Corollary~\ref{cor:lp-ellipsoid}
by an extra factor of $\sqrt{n/(n-k +1)}$ in front of $\sigma_{k+1}$ (this could lead to a worse bound when the infimum
over $k$ is attained at $k$ close to $n$).

\smallskip

Another natural approach was developed by the authors of \cite{HKNPU, HPS}.
Using the technique of the  low $M^*$-estimate (see Remark~\ref{M*}),
they proved the following theorem.

\medskip
\noindent
\textbf{Theorem B.} \ \cite[Theorem A]{HPS}
{\it There exist  constants $0<\tilde{c} <1<\widetilde{C}$ and $\gamma>0$ such that for every $1 \le n < N$ and every $p>1$,
\begin{equation}
	\label{thmA}
{\cal R}_n({\cal E}_p)
\le
\widetilde{C}
\sqrt{\frac{p'}{n}} \left( \sum_{j > \tilde{c} \, n} \sigma_j^{p'} \right)^{1/p'}
\end{equation}
with probability at least  $1 - e^{- \gamma n}$.
}

\medskip

For $p \ge 2$, let $\tilde{c}$ be the constant appearing in Theorem~B and choose $k = c' n$ with $c' > \tilde{c}$.
Then Corollary~\ref{cor:lp-ellipsoid} implies
\[
{\cal R}_n({\cal E}_p)
\le C_\eps \left( \sum_{j > c' \, n} \sigma_j^{\frac{2p}{p-2}} \right)^{1/2-1/p}
+
\frac{1}{\sqrt{n}} \left( \sum_{j > n} \sigma_j^{p'} \right)^{1/p'}.
\]
Let $\alpha = c' - \tilde{c}$. Since $\sigma$ is a decreasing sequence, a decomposition of the sums in blocks of size $\alpha n$ gives  that
\[
\frac{1}{\sqrt{n}} \left( \sum_{j > \tilde{c} \, n} \sigma_j^{p'} \right)^{1/p'}
\ge
\sqrt{\alpha} \left( \sum_{j > c' \, n} \sigma_j^{\frac{2p}{p-2}} \right)^{1/2-1/p}.
\]
Therefore the upper bound given by Corollary~\ref{cor:lp-ellipsoid} is always smaller (up to a constant) than the one given by
Theorem~B. However, if $p<2$, the two estimates are not comparable. For example (see below), in the case of a sequence $\{\sigma_i\}_{i=1}^N$ of polynomial decay, the upper bound from Theorem~B is better, while for exponential decay, the upper bound from Corollary~\ref{cor:lp-ellipsoid} leads to a  better (and sharp) estimate.

To illustrate the strength of our results, we apply them to three types of examples in detail. We start by considering the case of a sequence $\{\sigma_j\}_j$ with exponential decay
by taking $\sigma_j = q^j$ for some fixed $0<q<1$. While the classical
``low $M^*$-estimate" does not give any satisfactory
answer (see \cite{GiaMT2005} and Example~2.2 in \cite{GianMil1997}), Corollaries~\ref{cor:lp-ellipsoid} and \ref{cor:lp-ellipsoid-lower} imply a complete answer to this problem. This complements the work in \cite{HKNPU},
where the authors obtained only partial results in the case of ellipsoids (see Corollary~9 there).

In the second example, we study the case of polynomial
decay by considering $$\sigma_j = j^{-\alpha} (\log(j+1))^{-\beta}$$ for some parameters $\alpha>0$, $\beta\geq 0$.
 Such sequences appear naturally in Approximation Theory: in particular, they correspond to  the
approximation numbers of embeddings of Sobolev spaces into $L_q$ (we refer to  Chapter 4 of
\cite{DTU-CRM-Barcelona-2018} for the details).
 Our results illustrate the dichotomy that we emphasized in Theorem~\ref{cor:dichotomy}.
 Such a dichotomy (with corresponding bounds) was already established in the case $p=2$ in \cite{HKNPU}.
 We obtain bounds (that hold with   high probability) in the case of
 polynomial decay, quantifying the behavior of the random radius.
 Our bounds are  sharp for $p\ge 2$ and for most of the cases when $p< 2$. They improve upon results from \cite{HPS},
 where the authors considered only the case $\beta = 0$.  Propositions~\ref{caseII} and \ref{caseIII} provide
a complete answer to the conjecture  formulated in \cite{HPS},  where they asked if random information is useless or not in the cases
$1 < p < 2$, $\alpha = 1-1/p$, or
 $p>2$, $1/2< \alpha \le 1-1/p$.  We show that random information is indeed useless in these situations.

In the third example, we discuss cases where $\sigma_1=\sigma_2=...=\sigma_m$ and $\sigma_{m+1}=...=\sigma_N$ for some $m$, and
 $\sigma_1/\sigma_N$ is sufficiently large.

 We conclude this section by proving Theorem~\ref{cor:dichotomy}.

\subsection{Exponential decay of the sequence $\sigma$}
Let $q \in \, (0,1)$ and $\sigma = \{ \sigma_j\}_{j=1}^N$ be such that $\sigma_j = \sigma_1 \, q^j$ for all $j \ge 1$. We set $p \in (1, +\infty)$.
We say that $a \lesssim b$ if there exists a constant $C_{p, q} > 0$ depending only on $p$ and $q$ such that
$a \le C_{p, q} \, b$. We say that $a \gtrsim b$ if $b \lesssim a$, and that $a \approx b$ if $a \lesssim b$ and $b \lesssim a$.

\begin{proposition}
	\label{caseGeom}
 Let $\delta >0$ be small enough. Then for every $c \log(1/\delta) \le n \le N$ (where $c$ is a universal constant),  we have
\[
\P\left(\frac{1}{\sqrt{\log(1/\delta)}} \, \sqrt{n} \, \sigma_{n+1}
	\lesssim
	{\cal R}_n({\cal E}_p)
	\lesssim
	\frac{\sqrt{\log(1/\delta)}}{\delta}  \, \sqrt{n} \, \sigma_{n+1} \right) \ge 1 - 2 \delta - n e^{- \gamma n}.
\]
\end{proposition}
\noindent
\begin{proof}
We have $R({\cal E}_p) \approx \sigma_1.$
For the upper bound, we set $k = n$ and $\eps \approx \delta$, and apply Corollary~\ref{cor:lp-ellipsoid}. Obviously
\[
\left(\sum_{i>n} \sigma_i^{\frac{2p}{p-2}}\right)^\frac{p-2}{2p} \approx \sigma_{n+1}
\
\textrm{and}
\
\left(\sum_{i>n} \sigma_i^{p'}\right)^{1/p'} \approx \sigma_{n+1} .
\]
We deduce that for every $n < N$,  with probability greater than $1 - \delta - ne^{- \gamma n}$,
\[
{\cal R}_n({\cal E}_p)
\lesssim
\frac{\sqrt{\log(1/\delta)}}{\delta}  \, \sqrt{n} \, \sigma_{n+1}.
\]
For the lower bound, we set $t = 2 \sqrt{\log(1/\delta)}$ and apply Corollary~\ref{cor:lp-ellipsoid-lower}. Since
\[
\bigg( \sum_{i=1}^{n+1} \frac{1}{\sigma_i^p}\bigg)^{1/p} \approx \frac{1}{\sigma_{n+1}}
\quad
\mathrm{and}
\quad
\bigg(\sum_{i=1}^{n+1} \frac{1}{\sigma_i^{2p/(2-p)}}\bigg)^{1/p - 1/2} \approx \frac{1}{\sigma_{n+1}},
\]
we deduce that for every $n < N$,  with probability greater than $1 - e^{-n/8} - \delta$,
\[
{\cal R}_n({\cal E}_p)
\gtrsim
\frac{1}{\sqrt{\log(1/\delta)}} \, \sqrt{n} \, \sigma_{n+1}. \qedhere
\]
\end{proof}

\subsection{Polynomial decay of the sequence $\sigma$}
Let $\alpha>0$ and $\sigma = \{ \sigma_j\}_{j=1}^N$ be such that $\sigma_j = j^{-\alpha} \log^{-\beta}(j+1)$ for all $j \ge 1$. We focus on the case $p \in (1, +\infty)$, even though the computations for p=1 are also tractable.
Again, we use the notation $a \lesssim b,  a \approx b$ with the constants this time possibly depending on $\alpha$, $\beta$ and $p$.

For such sequences $\sigma$, it is possible to evaluate the parameters that come up in the study of high dimensional convex bodies, like the radius, $\ell$, $\ell_*$ or the Dvoretzky dimension.
By Lemma~\ref{geom-p-ell}, for every $p \in (1, +\infty)$, one has
\[
	\ell ({\cal E}_p) \approx  \sqrt{p} \, N^{1/p + \alpha} (\log(N+1))^{\beta},
	\quad
	\ell_*({\cal E}_p) \approx \left( \sum_{j=1}^N \sigma_j^{p'} \right)^{1/p'},
\]
while
\[
 \left\{
 \begin{array}{rll}
d_{{\cal E}_p} \approx	N, & R({{\cal E}_p}) \approx 1 & \mathrm{if} \ p < 2
	\\
d_{{\cal E}_p} \approx	p N^{2/p}, & R({{\cal E}_p}) \approx \left(\sum_j \sigma_j^{2p/(p-2)}\right)^{1/2 - 1/p} & \mathrm{if} \ p \ge 2
\end{array}
\right.\ .
\]
If $\alpha > 1/p'$, the series
$\sum \sigma_j^{p'}$ converges, hence $\ell({\cal E}_p) \ell_*({\cal E}_p) / N \gtrsim N^{\alpha - 1/p'} (\log(N+1))^{\beta}$. This means that
${\cal E}_p$ is not in a good position compared to what is called $M$-position \cite{Pisier1989}.
As in \cite{HPS}}, our result states that,
if $\sum \sigma_j^{p'}$ converges, then the random radius
${\cal R}_n({\cal E}_p)$ decays polynomialy fast to $0$ with $n$. We prove an exact rate of polynomial decay with explicit logarithmic factors.
The precise statements are given in Propositions  \ref{caseI} and \ref{caseIbis}.
In sharp contrast,
if $\sum \sigma_j^{p'}$ diverges, then randomness cannot help, and this can be abstractly confirmed as follows: in Theorem \ref{cor:dichotomy} we provide a function $f$ (depending on $\sigma$) going to infinity at infinity such that
for all $n \le f(N)$, ${\cal R}_n({\cal E}_p)$ is lower bounded by a constant
(and therefore cannot decay to zero). This is also explicitly corroborated in Propositions \ref{caseIII}
to \ref{caseIV} for $\sigma$ of polynomial decay.

A final remark is that for all $p \ge 2$, we prove sharp estimates on the random radius in any case. On the other hand, when $p \in (1,2)$, then in some cases the estimates we can obtain are not as sharp as we would like,
but the difference in the form of the upper and lower bound is rather small that the conclusion about the general behavior of the
random radius is still valid.

\smallskip
We start with a general lower bound on the random radius. It is consistent with the well-known fact that we recalled in the introduction too, that for very large codimensions, $n > N - d_{{\cal E}_p}$, the value of the random radius
is given by the Dvoretzky Theorem, and it is proportional to $\sqrt N / \ell({\cal E}_p)$.

\begin{lemma}[General lower bound]
	\label{lem:lowerbound-pol}
	For every $n < N$, one has
	\[
	{\cal R}_n ( {\cal E}_p)
	\gtrsim
	n^{1/2 - 1/p - \alpha} (\log(n+1))^{-\beta},
	\]	
	with probability greater than $1- 2e^{- \gamma n}$ if $p<2$ or greater than $1 - 2e^{-\gamma p n^{2/p}}$ if $p \ge 2$.
\end{lemma}

\noindent
{\bf Proof.}
One has
\[
\left(\sum_{i=1}^{n+1} \sigma_i^{-p} \right)^{1/p} \approx n^{1/p} \frac{1}{\sigma_{n+1}} \approx n^{1/p + \alpha} (\log(n+1))^{\beta}
\]
and for $p < 2$,
\[
\left(\sum_{i=1}^{n+1} \sigma_i^{-2p/(2-p)} \right)^{1/p - 1/2} \approx n^{1/p - 1/2} \frac{1}{\sigma_{n+1}} \approx n^{1/p -1/2 + \alpha} (\log(n+1))^{\beta}.
\]
An application of Corollary \ref{cor:lp-ellipsoid-lower} with $t=\sqrt{n}$ when $1<p < 2$ or $t = \sqrt{p}\,n^{1/p}$ when $p\geq 2$ gives the result.
\qed

\bigskip
Let $p \in (1, \infty)$.
We divide the region $\alpha >0, \beta \in \RR$  into 7 regions, which leads to the statement of 7 propositions.  These regions are delimited by the behavior of the series $\sum \sigma_j^{p'}$ and  $\sum \sigma_j^{2p/(p-2)}$, that control the diameter of
${\cal E}_p$ when $p >2$. In all propositions, $\gamma$ is an absolute positive constant
(mainly coming from the probabilistic estimate in Theorem~\ref{MainThm}); let us also reiterate that in all these propositions the bounds that we obtain for $p\ge 2$ are sharp.

\subsubsection{Good decay of the random radius}
\label{Good}
In the first two cases, the series  $\sum \sigma_j^{p'}$ converges. We establish sharp decay of the random radius, with
overwhelming probability.
\begin{proposition}
	\label{caseI}
	Assume that $\alpha  > 1- 1/p$ and $\beta \in \RR$. Then
	\[
	 \forall n < N, \
	\P \bigg( {\cal R}_n ( {\cal E}_p ) \approx n^{1/2- 1/p -\alpha} \log^{-\beta}(n+1) \bigg)
	\ge
	\left\{
	\begin{array}{ll}
		1 - 2 e^{-\gamma n} & \ \mathrm{if} \ p \le 2
		\\
		1 - 2 e^{-\gamma p n^{2/p}} & \ \mathrm{if} \ p > 2
	\end{array}
	\right..
	\]
\end{proposition}

\medskip
\noindent
{\bf Proof.}  As $\alpha > 1-1/p$, $R({\cal E}_p) \approx 1$.
 Theorem~B implies  that for all $n < N$,  with probability greater than $1 - e^{-\gamma n}$,
\[
 {\cal R}_n ( {\cal E}_p)
\lesssim
\sqrt{\frac{p'}{n}} \left( \sum_{j > \tilde{c} \, n} \sigma_j^{p'} \right)^{1/p'}
\lesssim
n^{1/2- 1/p -\alpha} \log^{-\beta}(n+1)\,.
\]

The lower bound follows from Lemma \ref{lem:lowerbound-pol}.
Thus we  proved a tight bound on the decay of the random radius.
\qed

\begin{proposition}
	\label{caseIbis}
Assume that $\alpha = 1 - 1/p$ and that $\beta > 1 - 1/p$.
\\
If $p \ge 2$ then for all $n < N$, one has
\[
{\cal R}_n ( {\cal E}_p ) \approx n^{-1/2} \log^{-\beta}(n+1) \cdot \left(\max\left\{1, \, \log(N/n)  \frac{\log n}{\log N}\right)\right\}^{1-1/p} ,
\]
with probability at least $1 - 3e^{-\gamma n}$ if $n \lesssim N$, and  at least  $1 -e^{-\gamma n} -2 e^{-\gamma p n^{2/p}}$
for the remaining values of $n$.
\\
If $ p<2$ then for all $n < N,$ one has
\[
n^{-1/2} \log^{-\beta}(n+1) \lesssim {\cal R}_n ( {\cal E}_p )
\lesssim
n^{-1/2} \log^{-\beta}(n+1) \cdot \left(\max\left\{1, \, \log(N/n)  \frac{\log n}{\log N}\right\}\right)^{1-1/p} ,
\]
with probability at least  $1 - 2 e^{-\gamma n}$.
\end{proposition}

\begin{remark}
Since for every $n \le \sqrt{N},$ $\log N \ge \log(N/n) \ge (\log N)/2$, one has for every $p  \ge 2$ and every $n < \sqrt{N}$,
\[
{\cal R}_n ( {\cal E}_p ) \approx n^{-1/2} (\log(n+1))^{-\beta} \cdot (\log(n+1))^{1 - 1/p}
\]
with overwhelming probability. 	This generalizes and quantifies  Corollary 7 from \cite{HKNPU} to all $p$-ellipsoids and  all dimensions $N$.
The important point is that we provide a strong lower bound and establish  the precise decay of the random radius
 for every $p \in (1, \infty)$.
\end{remark}

The following lemma is obtained by standard calculus tools.
\begin{lemma}
	\label{lem:computation1}
	For any $y > x > 1$, one has
	\[
	\int_{\log x}^{\log y} u^{- \lambda} du \approx
	\left\{
	\begin{array}{ll}
		\log(y/x) \cdot (\log(y))^{- \lambda}  & \quad \mathrm{if} \ \lambda < 1
		\\
		\log(y/x) \cdot (\log(x))^{- \lambda} \cdot \frac{\log(x)}{\log(y)}  & \quad \mathrm{if} \ \lambda > 1
	\end{array}
	\right..
	\]
	Let $1 \le k \le N$. Define
	\[
      \phi_N(k) =  \log(N/k)  \frac{\log(k)}{\log(N)}\,.
	\]
	Then $\phi_N$ is increasing on $[1, \sqrt{N}]$ and decreasing on $[\sqrt{N}, N]$. Moreover,
$$\phi_N(2) = \phi_N(N/2) \approx 1\quad \quad \mbox{ and } \quad \quad \phi_N(\sqrt N) = (\log(N))/4.$$
\end{lemma}

\noindent
{\bf Proof of Proposition \ref{caseIbis}.} One has $\alpha = 1-1/p$ and $\beta > 1-1/p$. Therefore $R({\cal E}_p) \approx 1$.
Since $\beta p' > 1$ and $\alpha p' = 1$,  one gets from Lemma \ref{lem:computation1} that
\begin{equation}
	\label{eq:comp2}
	\left(\sum_{j=k+1}^N \sigma_j^{p'} \right)^{1/p'}
	=  \left(\sum_{j=k+1}^N \frac{1}{j} \cdot (\log(j))^{- \beta p'} \right)^{1/p'}
	\approx (\log(k))^{-\beta} (\phi_{N}(k))^{1/p'},
\end{equation}
using comparison between series and integrals, and observing that $\phi_{N+1}(k+1) \approx \phi_N(k)$.
It can be seen that for a fixed constant $c < 1$,
$\phi_N(cn) \approx \max(1, \phi_N(n))$.
Therefore, from Theorem~B, with probability greater than $1 - e^{-\gamma n}$, one has
\[
{\cal R}_n ( {\cal E}_p)
\lesssim n^{-1/2} (\log(n))^{-\beta} (\max(1, \phi_N(n)))^{1/p'}.
\]
This shows the upper bound in both cases.

Let $p \ge 2$. Since $\alpha = 1 - 1/p$, then $\alpha \cdot 2p/(p-2) >1$ and for all $k \le N/2$, one has
\[
\left(\sum_{j=k+1}^N \sigma_j^{2p/(p-2)} \right)^{1/2 - 1/p} \approx k^{-1/2} (\log(k))^{- \beta}.
\]
Now, let $n$ be such that $\phi_N(n) \gtrsim 8^{p'}$. Then by
using \eqref{eq:comp2}, one gets for $k=n$,
\[
\left(\sum_{j=n+1}^N \sigma_j^{p'} \right)^{1/p'} \ge 8 \sqrt{n} \left(\sum_{j=n+1}^N \sigma_j^{2p/(p-2)} \right)^{1/2 - 1/p} .
\]
From Lemma \ref{lem:computation1} and the variations of $\phi_N$, the condition $\phi_N(n) \gtrsim 8^{p'}$ is satisfied if and only if $1 \lesssim n \lesssim N$.
Thus, as we just checked, assumption  \eqref{eq:hypLower} of Corollary \ref{LowerEllpThm} also follows for such $n$ (and $k=n$), and hence
\[
{\cal R}_n ( {\cal E}_p ) \gtrsim n^{-1/2} (\log(n))^{-\beta} (\phi_N(n))^{1/p'},
\]
with probability greater than $1 - 2e^{-2n}$.

For the remaining values of $n$ or when $1 < p < 2$, we apply Lemma \ref{lem:lowerbound-pol} which gives a general lower bound.
\qed

\subsubsection{Bad behavior of the random radius}
In the remaining cases, we prove that the random radius can never shrink below at least a constant value
independently of the dimension of the ambient space.

\begin{proposition}
	\label{caseIII}
	Assume that $1/2 - 1/p < \alpha  < 1- 1/p$, and $\beta \in \RR$.
	\\
	If $p \ge 2$ then for all  $n < N$, one has
	\[
	{\cal R}_n ( {\cal E}_p )
	\approx
	\min \left\{ 1,
	n^{-1/2} N^{1-1/p - \alpha} (\log(N+1))^{-\beta} \right\}
	\]
	with probability  greater than $1 - e^{-\gamma n} - 2 e^{-\gamma p n^{2/p}}$.
	\\
	If $p<2$ then for all $n<N$
	\[
	n^{1/2 - 1/p - \alpha} (\log(n+1))^{-\beta}
	\lesssim
	{\cal R}_n ( {\cal E}_p )
	\lesssim
	\min \left( 1,
	n^{-1/2} N^{1-1/p - \alpha} (\log(N+1))^{-\beta}
	\right)
	\]
	with probability at least $1 - 3 e^{-\gamma n}$. And, for all $n \lesssim N^{2/p'}$, one has
	\[
	{\cal R}_n ( {\cal E}_p )
	\approx
	\min \left( 1,
	n^{-1/2} N^{1-1/p - \alpha} (\log(N+1))^{-\beta} \right)
	\]
	with probability at least  $1 - 3 e^{-\gamma n}$.

	Moreover
	\[
	\forall n \lesssim N^{2 (1-1/p - \alpha)} (\log(N+1))^{2\beta}, \quad
	\P \big( {\cal R}_n ( {\cal E}_p ) \approx 1  \big)
	\ge
	1 - 2 e^{- 2n}.
	\]	
\end{proposition}

\noindent
{\bf Proof.} Since $1/2 - 1/p < \alpha$, the series
$\sum \sigma_j^{2p/(p-2)}$ is convergent and $R({\cal E}_p) \approx 1$.
Since $\alpha < 1/p'$, one has for all $k \le N/2$,
\[
\left( \sum_{j=k+1}^{N} \sigma_j^{p'} \right)^{1/p'}
=
\left( \sum_{j=k+1}^{N} j^{-\alpha p'} (\log(j+1))^{-\beta p'} \right)^{1/p'}
\approx
N^{1/p' - \alpha} (\log(N+1))^{-\beta}.
\]
Taking $k=\tilde{c}\,n$ in the latter and
using Theorem~B, one gets that with probability greater than $1 - e^{-\gamma n}$,
\[
{\cal R}_n ( {\cal E}_p )
\lesssim
\min\left(1,
n^{-1/2} N^{1/p' - \alpha} (\log(N+1))^{-\beta}
\right)
\]
and the upper bound follows.

Let $p \ge 2$. Since $\alpha \cdot 2p/(p-2) > 1$, one has
\[
\left(\sum_{j=k+1}^N \sigma_j^{2p/(p-2)} \right)^{1/2 - 1/p} \approx k^{1/2-1/p} \, k^{-\alpha} (\log(k+1))^{-\beta}
= k^{-1/2} \, k^{1/p' -\alpha} (\log(k+1))^{-\beta}.
\]
Therefore, if $n \lesssim N$, assumption~\eqref{eq:hypLower} is satisfied for
$k \approx N$. An application of Corollary~\ref{LowerEllpThm} shows that
\[
{\cal R}_n ( {\cal E}_p )
\gtrsim
\min\left(1,
n^{-1/2} N^{1/p' - \alpha} (\log(N+1))^{-\beta}
\right),
\]
with probability greater than $1 - 2e^{-2n}$.

Let $p<2$. Assumption~\eqref{eq:hypLower} is satisfied whenever $k$ is such that
\[
k \le N/2
\quad  \mathrm{and} \quad
N^{1/p' - \alpha} (\log(N+1))^{-\beta} \gtrsim 8 \sqrt{n} k^{-\alpha} (\log(k+1))^{-\beta}.
\]
Therefore, if $n \lesssim N^{2/p'}$ we choose $k \approx N$ so that the assumption is satisfied.
Then an application of Corollary~\ref{LowerEllpThm} completes the proof.

The ``Moreover" part follows immediately. If $n \lesssim N^{2 (1-1/p - \alpha)} (\log(N+1))^{2\beta}$, one has
\[
\min\left\{1,
n^{-1/2} N^{1/p' - \alpha} (\log(N+1))^{-\beta} \right\} \approx 1
\]
which corresponds to ${\cal R}_n({\cal E}_p)$ for $p \ge 2$ and, since $n \lesssim N^{2/p'}$, for $p<2$ as well.
\qed

\begin{proposition}
	\label{caseII}
	Assume that $\alpha  = 1 - 1/p$ and $\beta < 1-1/p$.
	\\
		If $p \ge 2$ then, for all $n < N$, one has
	\[
	{\cal R}_n ( {\cal E}_p )
	\approx
	\min\left\{1,
	n^{-1/2} \log^{-\beta}(N+1)\big( \log^\alpha(eN/n)\big) \right\}
	\]
	with probability  at least $1 - e^{-\gamma n} - 2 e^{-\gamma p n^{2/p}}$.
	\\
	If $p < 2$ then, for all $n < N$, one has
	\[
	n^{-1/2} \log^{-\beta}(n+1)
	\lesssim
	{\cal R}_n ( {\cal E}_p )
	\lesssim
	\min\left\{1,
	n^{-1/2} \log^{-\beta}(N+1)\big( \log^\alpha(eN/n)\big) \right\}
	\]
	with probability at least  $1 - 3e^{-\gamma n}$.

	Moreover
	\[
	\forall n \lesssim (\log(N+1))^{2 (\alpha - \beta)}, \
	\P \big( {\cal R}_n ( {\cal E}_p ) \approx 1  \big)
	\ge
	1 - 2e^{- 2n}.
	\]
\end{proposition}
Observe that if $n$ is proportional to $N$, the result is also sharp for $p<2$. In any case, once the upper bound reaches the level of an absolute constant, then it does give the correct answer as the ``moreover part" shows. This happens for a big enough range of $n$ relatively to $\log (N+1)$, since $\alpha - \beta > 0$.

\medskip

\noindent
{\bf Proof.}  We have $\alpha = 1-1/p$ and $\beta < 1-1/p$. Therefore the series
$\sum_{j> 1} \sigma_j^{2p/(p-2)}$ is convergent and
$R({\cal E}_p) \approx 1$.

Since $\beta p' < 1$, one has by Lemma \ref{lem:computation1},
\[
\left( \sum_{j = k+1}^N \sigma_j^{p'} \right)^{1/p'} = \left( \sum_{j = k+1}^N \frac{1}{j} (\log(j))^{-\beta p'} \right)^{1/p'} \approx (\log(eN/k))^{1/p'} (\log(N))^{- \beta}.
\]
Taking $k=\tilde{c}\,n$ in the latter and
using Theorem~B,  one gets that with probability greater than $1 - e^{-\gamma n}$,
\[
{\cal R}_n ( {\cal E}_p )
\lesssim
\min\left(1,
n^{-1/2} \log^{-\beta}(N+1)\big( \log^\alpha(eN/n)\big)\right).
\]
This implies the upper bound.

Let $p \ge 2$. We have again
\[
\sqrt{n} \left(\sum_{j=k+1}^N \sigma_j^{2p/(p-2)} \right)^{1/2 - 1/p} \approx n^{1/2} k^{-1/2} (\log(k))^{- \beta}
\]
and
\[
\left( \sum_{j = k+1}^N \sigma_j^{p'} \right)^{1/p'}  \approx (\log(eN/k))^{1/p'} (\log(N))^{- \beta}.
\]
Since $1/p' > \beta$ and $n < N$, one has
\[
(\log(eN/k))^{1/p'} \left(\frac{\log(n)}{\log(N)}\right)^{\beta} \ge
\left(\log(eN/k) \frac{\log(n)}{\log(N)}\right)^{1/p'} = \phi_N(n)^{1/p'}.
\]
As in the proof of Proposition \ref{caseIbis}, let us consider $n$ such that $\phi_N(n) \gtrsim 8^{p'}$, and note that for $k=n$
\[
\left(\sum_{j=n+1}^N \sigma_j^{p'} \right)^{1/p'} \ge 8 \sqrt{n} \left(\sum_{j=n+1}^N \sigma_j^{2p/(p-2)} \right)^{1/2 - 1/p} .
\]
Applying Corollary~\ref{LowerEllpThm}, we deduce that, for all $1 \lesssim n \lesssim N$,
\[
{\cal R}_n ( {\cal E}_p ) \gtrsim n^{-1/2} (\log(N))^{-\beta} (\log(eN/n))^{1/p'},
\]
with probability greater than $1 - 2e^{-2n}$. This is the requested lower bound.

For the remaining values of $n$ or when $1 < p < 2$, we apply Lemma \ref{lem:lowerbound-pol} which gives a general lower bound.

We now prove the ``moreover part" of the Proposition. Assume $n \lesssim (\log(N))^{2 (\alpha - \beta)}$.
The case $p\ge 2$ is already proven since the estimate gives ${\cal R}_n ( {\cal E}_p ) \approx 1$ in this range of integers $n$.
It remains to deal with $p<2$.
Since
\[
\left( \sum_{j=1}^{N} \sigma_j^{p'} \right)^{1/p'} \approx \big( \log N \big)^{1/p' - \beta} = \big( \log N \big)^{\alpha - \beta}
\qquad \mathrm{and}
\qquad
\sigma_1 \approx 1
\]
it is clear that assumption~\eqref{eq:hypLower} is satisfied with $k=1$ as soon as $n \lesssim (\log N)^{2(\alpha - \beta)}$.
By Corollary~\ref{LowerEllpThm}, one gets
\[
{\cal R}_n({\cal E}_p) \gtrsim 1,
\]
with probability greater than $1 - 2e^{-2n}$.
By also recalling that ${\cal R}_n ( {\cal E}_p) \le R ( {\cal E}_p) \approx 1$, we complete  the proof.
\qed

\begin{proposition}
	\label{caseIIbis}
	Assume that $\alpha  = \beta = 1 - 1/p$.
	\\
	If $p \ge 2$ then, for all $n < N$, one has
	\[
	{\cal R}_n ( {\cal E}_p )
	\approx
	\min\left\{1,
	n^{-1/2} \left( \log\left(\frac{\log N}{\log n}\right)\right)^{\beta} \right\},
	\]
	with probability  at least $1 - e^{-\gamma n} - 2 e^{-\gamma p n^{2/p}}$.
	\\
		If $p < 2$ then, for all $n < N$, one has
	\[
	n^{-1/2} \log^{-\beta}(n+1)
	\lesssim
	{\cal R}_n ( {\cal E}_p )
	\lesssim
	\min\left\{1,
	n^{-1/2} \left( \log\left(\frac{\log N}{\log n}\right)\right)^{\beta} \right\},
	\]
	with probability  at least $1 - 3 e^{-\gamma n}$.
	
	Moreover
	\[
	\forall n \lesssim (\log(\log(N)))^{2 \alpha}, \
	\P \big( {\cal R}_n ( {\cal E}_p ) \approx 1  \big)
	\ge
	1 - 2e^{- 2n}.
	\]	
\end{proposition}

\noindent
{\bf Proof.} We proceed along the same lines as in the proof of Proposition~\ref{caseII},
replacing the evaluation of $(\sum \sigma_j^{p'})^{1/p'}$ by the following estimate
\[
\left( \sum_{j=k+1}^{N} \sigma_j^{p'} \right)^{1/p'}
\approx
\left( \log\left(\frac{\log N}{\log n}\right)\right)^{1/p'}.
\]
Then we look at the function
\[
\psi_N(n) = \log\left(\frac{\log N}{\log n} \right) \cdot \log n
\]
which is increasing on $[2, N^{1/e}]$, decreasing on $[N^{1/e}, N]$, and satisfies
$\psi_N(2) \approx \psi_N(N/2) \approx 1$. With this observation in mind, the argument becomes a repetition of the proof of Proposition~\ref{caseII}.
\qed

\medskip
Since $\alpha \ge 0$, the next 3 cases concern only values of $p$ larger than 2.
\begin{proposition}
	\label{caseIIIbis}
	Assume that $\alpha = 1/2 - 1/p$ and $\beta > 1/2-1/p$.
\\	
	Then for all $n \lesssim N  (\log(N+1))^{ - 2 \beta}$, one has
	\[
	{\cal R}_n ( {\cal E}_p )
	\approx 1,
	\]
	for all $n$ such that $N  (\log(N+1))^{ - 2 \beta} \lesssim n \lesssim N$, one has
	\[
	{\cal R}_n ( {\cal E}_p )
	\approx
	n^{-1/2} N^{1/2} (\log(N+1))^{-\beta},
	\]
	and for all $n \gtrsim N $, one has
	\[
	{\cal R}_n ( {\cal E}_p )
	\approx
	(\log(N+1))^{-\beta},
	\]
	with probability at least $1 - 3 e^{-\gamma n}$ in the  first two cases  and $1 - 3 e^{-\gamma p n^{2/p}}$ in the last one.
\end{proposition}

\noindent
{\bf Proof.}
Since $\alpha = 1/2 - 1/p$ and $\beta > 1/2-1/p$, the series $\sum \sigma_j^{2p/(p-2)}$ is convergent and $R({\cal E}_p) \approx 1$. The proof essentially repeats the proof of Proposition~\ref{caseIII}. Since $\alpha < 1/p'$, one has for all $k \le N/2$,
\[
\left( \sum_{j=k+1}^{N} \sigma_j^{p'} \right)^{1/p'}
=
\left( \sum_{j=k+1}^{N} j^{-\alpha p'} (\log(j+1))^{-\beta p'} \right)^{1/p'}
\approx
N^{1/2} (\log(N+1))^{-\beta}.
\]
Taking $k=cn$ in the latter and
using Theorem~B,  one gets that with probability  at least  $1 - e^{-\gamma n}$,
\[
{\cal R}_n ( {\cal E}_p )
\lesssim
\min\left(1,
n^{-1/2} N^{1/2} (\log(N+1))^{-\beta}
\right).
\]
Since $\beta \cdot 2p/(p-2) > 1$, we obtain from Lemma \ref{lem:computation1} that
\[
\left(\sum_{j=k+1}^N \sigma_j^{2p/(p-2)} \right)^{1/2 - 1/p}
=
\left(\sum_{j=k+1}^N \frac{1}{j} (\log(j))^{- \beta \cdot 2p/(p-2)} \right)^{1/2 - 1/p}
\approx
(\log(k))^{-\beta} (\phi_N(k))^{1/2 - 1/p}.
\]
If $n \lesssim N$, we see from Lemma \ref{lem:computation1} that we can choose $k = N/2$ such that $\phi_N(k) \approx 1$ and
\[
\left( \sum_{j=k+1}^{N} \sigma_j^{p'} \right)^{1/p'}
\ge 8 \sqrt{n} \left(\sum_{j=k+1}^N \sigma_j^{2p/(p-2)} \right)^{1/2 - 1/p}.
\]
An application of Corollary \ref{LowerEllpThm} yields the first two cases of the proposition. For the remaining values of $n$, which are $\gtrsim N$, it is enough to apply
Lemma \ref{lem:lowerbound-pol}.

\begin{proposition}
	\label{caseIV}
	Assume that either $(\alpha = 1/2 - 1/p$ and $\beta \le 1/2 - 1/p)$ or that $(\alpha  <  1/2- 1/p$ and $\beta \in \RR)$.
	Then
	\[
	{\cal R}_n ( {\cal E}_p ) \approx R({\cal E}_p)	
	\]
	with probability  at least  $1 - 2e^{-2 n}$ for $n \lesssim N$ and  at least  $1 - 2 e^{-\gamma p n^{2/p}}$ for $n \gtrsim N$.
\end{proposition}

\noindent
{\bf Proof.}
Observe that for every $k \le N/2$, one has
\[
\left( \sum_{j=k+1}^{N} \sigma_j^{p'} \right)^{1/p'}
\approx
N^{1 - 1/p-\alpha} (\log(N+1))^{-\beta}.
\]
In both cases, the series $\sum \sigma_j^{2p/(p-2)}$ is divergent.

If $\alpha < 1/2 - 1/p$, then $R({\cal E}_p) \approx N^{1/2 - 1/p - \alpha } (\log(N+1))^{-\beta}$ and for all $k \le N/2$, one has
\[
\left(\sum_{j=k+1}^N \sigma_j^{2p/(p-2)} \right)^{1/2 - 1/p}
\approx
N^{1/2 - 1/p- \alpha} \, (\log(N+1))^{-\beta}.
\]
Therefore, whenever $n \lesssim N$, assumption~\eqref{eq:hypLower} is satisfied with $k = N/2$.
Since
\[
\left(\sum_{j=1}^{N/2} \sigma_j^{2p/(p-2)} \right)^{1/2 - 1/p}
\approx
N^{1/2 - 1/p -\alpha} \, (\log(N+1))^{-\beta},
\]
an application of Corollary~\ref{LowerEllpThm} shows that, with probability greater than $1 - 2e^{-2n}$,
\begin{align*}
	{\cal R}_n ( {\cal E}_p ) & \gtrsim
	\min\left(N^{1/2 - 1/p -\alpha} \, (\log(N+1))^{-\beta}, n^{-1/2} N^{1 - 1/p- \alpha} (\log(N+1))^{-\beta} \right)
	\\
	&  =
	N^{1/2 - 1/p -\alpha} \, (\log(N+1))^{-\beta},
\end{align*}
which is proportional to  the radius of ${\cal E}_p$.

The case $\alpha = 1/2 - 1/p$ and $\beta \le 1/2 - 1/p$ is treated similarly.

The remaining values of $n$ are
covered by Lemma \ref{lem:lowerbound-pol}.
\qed

\subsection{Two-valued sequence $\sigma$}
Fix $1\leq m< N$. Consider a sequence $\sigma$ satisfying
$$\sigma_1=\sigma_2=\ldots=\sigma_m> \sigma_{m+1}= \ldots=\sigma_N>0$$
and denote $A=\sigma_1/\sigma_N>1$.
For the sake of clarity, we consider the case $p=2$ only.
As in the previous section we split our results in a few propositions.
Note that  we deterministically have
$$
      \sigma_{n+1}\leq   {\cal R}_n({\cal E}_2)\leq \sigma_1.
$$
In the first proposition we discuss several cases when ${\cal R}_n({\cal E}_2)\approx \sigma_1$,
in particular showing that the random diameter can be much larger than the smallest one. Note that
if $n<m$ then  we deterministically have
$$
     {\cal R}_n({\cal E}_2)\geq \sigma_{n+1}= \sigma_1,
$$
therefore below we consider only the case $n\geq m$.

\begin{proposition}	\label{2-v-1}
Let $1\leq m\leq n<N$.
In the following two cases ${\cal R}_n$ is of the order of $\sigma_1$.

\smallskip

\noindent
(1) If
$$
   A^2\leq \frac{n}{n-m+1},
$$
 then for every $s\geq 1$ one has
 $${\cal R}_n({\cal E}_2) \geq  \sigma_1/(4(s+1))$$
  with probability at least
  $$1-\exp{(-n/8)}-\exp{(-2ns^2/A^2)}\geq 1 - e^{-  n/8} - e^{-  2s^2(n+1-m)}.$$

\smallskip

\noindent
(2) If
$$
   A^2\leq \frac{N-1}{64 n},
$$
then  with probability at least  $1 - 2e^{- 2 n}$ one has
${\cal R}_n({\cal E}_2) \geq  \sigma_1/2\,$.
\end{proposition}

\noindent
{\bf Proof.} Each of the two assumptions above allows us to use one of Corollaries 3.4 and 3.5, which leads to the claimed bounds as follows.

\noindent
{\bf Case 1. $A^2=\sigma_1^2/\sigma_N^2\leq n/(n-m+1).$}
We have
$$
 \sum_{i=1}^{n+1} \frac{1}{\sigma_i^2} = \frac{m}{\sigma_1^2} +  \frac{n+1-m}{\sigma_N^2}
 \leq \frac{m+n}{\sigma_1^2}\leq \frac{2n}{\sigma_1^2}\,.
$$
Corollary~\ref{cor:lp-ellipsoid-lower} implies that
 for every $t>0$, with probability at least $1 - e^{-  n/8} - e^{-t^2/ 2}$,
\[
 {\cal R}_n({\cal E}_2 ) \ge
 \frac{\sqrt n}{2 \left( \displaystyle \sqrt{2} \bigg( \sum_{i=1}^{n+1} \frac{1}{\sigma_i^2}\bigg)^{1/2} + \frac{t}{\sigma_{n+1}} \right)}\,.
\]
Choose $t=2s\sqrt{n} /A$. Then
\[
{\cal R}_n({\cal E}_2 ) \ge \frac{\sqrt n  \, \sigma_1  }{4\, \sqrt{n} + 4s\sqrt{n} }\geq
\frac{ \sigma_1  }{4(1+s) }
\]
with probability at least $1 - e^{-  n/8} - e^{-  2s^2n/A^2}\geq 1 - e^{-  n/8} - e^{-  2s^2(n+1-m)}$.

\smallskip

\noindent
{\bf Case 2.  $A^2=\sigma_1^2/\sigma_N^2\leq (N-1)/(64n).$ }
In this  case, we use Corollary~\ref{LowerEllpThm} which implies that
for $p=2$, with probability at least  $1 - 2e^{- 2 n}$,
one has
\[
{\cal R}_n({\cal E}_2 )
\ge
\Bigg(
\frac{1}{\sigma_1} +
\frac{8 \, \sqrt{n}}{ \left( \sum_{j>1} \sigma_j^{2} \right)^{1/2}}
\Bigg)^{-1}
\]
provided that
$
\sum_{j>1} \sigma_j^{2}
\ge
64 n \,  \sigma_{2}^2.$
Indeed,
$$
 \sum_{j=2}^N \sigma_j^{2}\geq (N-1)\sigma_N^2\geq 64n \sigma_1^2,\,
$$
therefore Corollary~\ref{LowerEllpThm} yields the result.
\qed

\smallskip

Next we provide a general lower bound when $n\geq m$ and $A$ is large enough.
As its proof essentially repeats word for word the proof of Case~1 of Proposition~\ref{2-v-1}
(with the choice $t=2s\sqrt{n+1-m}$), we omit it.

\begin{proposition}	\label{2-v-2}
Let $1\leq m\leq n<N$ and assume that
$$
    A^2\geq \frac{n}{n-m+1}\,.
$$
 Then for every $s\geq 1$ one has
 $$
   {\cal R}_n({\cal E}_2) \geq
   \sqrt{\frac{  n  }{  n+1-m}}\, \,
   \frac{  \sigma_N  }{  4 (s+1)}
 $$
    with probability at least
$$
  1 - e^{-  n/8} - e^{-  2s^2(n+1-m)}.
$$
\end{proposition}

Finally we provide the upper bound, which is an immediate consequence of
Corollary~\ref{cor:lp-ellipsoid} applied with $k=m$ (and of Proposition~\ref{2-v-2}
for the ``in particular" part).

\begin{proposition}	\label{2-v-3}
Let $1\leq m\leq n<N$.
 Let  $C\geq 2$ and $\gamma$ be constants from Corollary~\ref{cor:lp-ellipsoid}.
Then for every $0 < \eps \le 1/(2C)$
with probability at least
\[
1 - 3 (C \eps)^{n-m+1} - m e^{- \gamma n}
\]
one has
\[
{\cal R}_n({\cal E}_2) \le \frac{2 \sqrt{2}}{\eps} \,\, \sqrt{\frac{n}{n-m +1}}\,  \,
\left( 4 \log(1/\eps) + \sqrt{\frac{N-m}{n-m+1}}\right) \sigma_{N}.
\]
In particular, if
$$
   A^2\geq \frac{n}{n-m+1}\quad \quad \mbox{and} \quad \quad
   N\leq C_1(n-m+1)+m
$$
for some absolute positive constant $C_1$ and, say, $\eps=1/2C$, then
$$
   {\cal R}_n({\cal E}_2) \approx \sqrt{\frac{n}{n-m +1}}\,   \sigma_{N}
$$
with large probability.
\end{proposition}

\begin{remark}
Note that our propositions do not  provide a sharp answer for
the interesting case when  $N=2n=2m$  and $A=\sigma_1/\sigma_N> \sqrt{n}$.
Then we only get
$$
  c\sqrt{n} \sigma_N \leq {\cal R}_n \leq C n \sigma_N.
$$
In particular, if $A\leq n$ the upper bound is trivial.
It would be interesting to obtain sharp bounds for this case.
\end{remark}

\subsection{Proof of Theorem~\ref{cor:dichotomy}}

The first case follows from Theorem~B (see the beginning of this section). Indeed, if
$ \sum_j \sigma_j^{p'}$ is convergent, then the sequence
\[
\varepsilon_n = \left( \sum_{j > \tilde{c}\,n} \sigma_j^{p'} \right)^{1/p'}
\]
goes to zero as $n$ goes to infinity (where $\tilde{c}$ here, as well as $\widetilde{C}$ and $\gamma$ in the next sentence, are the absolute constants fixed by Theorem B). From \eqref{thmA}, one has for  every $1 \le n < N$,
\[
\sqrt{n} \,{\cal R}_n({\cal E}_p) \le \widetilde{C} \sqrt{p'} \, \eps_n
\]
with probability greater than $1 - e^{-\gamma n}$.

\medskip
The second case is a consequence of Corollary \ref{LowerEllpThm}. Since $ \sum_j \sigma_j^{p'}$ is divergent, the sequence
$$
  S_N := \bigg(  \displaystyle \sum_{j = 2}^N \sigma_j^{p'} \bigg)^{1/p'}
$$
goes to infinity  as $N \to \infty$.  If $1 < p \le 2$, we define
$$
  f(N) := S_N^2 / (64 \sigma_1^2)\leq S_N^2 / (64 \sigma_2^2) \leq N,
$$
 and we observe that, for $ n \le f(N)$, the set $I_p$ introduced in Corollary~\ref{LowerEllpThm}
  contains the integer $k = 1$. Moreover, the condition $n \le f(N)$ implies that $\sigma_1 \leq S_N/(8\sqrt{n})$. Therefore,
with probability at least  $1 - 2e^{- 2 n}$ one has
\[
{\cal R}_n({\cal E}_p )
\geq
\frac{1}{2} \,  \min\Bigg\{ r_{1, p}, \,
	\frac{S_N}{8 \, \sqrt{n}} \Bigg\}
\geq  \frac{\sigma_1}{2}.
\]
If $p > 2$, we define $f(N) := S_N^2 / (64 A_N^2)$, where
\[A_N := \bigg( \displaystyle\sum_{j = 1}^N \sigma_j^{2p/(p-2)} \bigg)^{1/2 - 1/p}\geq \bigg( \displaystyle\sum_{j = 2}^N \sigma_j^{2p/(p-2)} \bigg)^{1/2 - 1/p}\geq \sigma_2\]
(here and below, just as in Corollary~\ref{LowerEllpThm}, $p/(p-2)$ is understood as $1/(1-2/p)=1$ when $p=\infty$).
Therefore again $f(N)\leq N$. Moreover, since
$$
  \frac{2p}{p-2} = \frac{p^2}{(p-1)(p-2)}+ \frac{p}{p-1} =\frac{p' \, p}{p-2} +p'
  \quad \mbox{ and } \quad \sigma_j\leq \sigma_1,
$$
 one has
\[
A_N \le
\sigma_{1}^{p'/2} \bigg( \sum_{j = 1}^N \sigma_j^{p'} \bigg)^{1/2 - 1/p}
\le \sigma_{1}^{p'/2}  S_N^{1-p'/2}\left(1+\tfrac{\sigma_1}{S_N}\right)^{1-p'/2}.
\]
This implies  that
\[
f(N) \ge \frac{S_N^{p'}} {64 \sigma_{1}^{p'} \left(1+\tfrac{\sigma_1}{S_N}\right)^{2-p'}},
\]
which tends to infinity as $N$ goes to infinity. Similarly to the previous case,
if $n \le f(N)$, then the set $I_p$ introduced in Corollary~\ref{LowerEllpThm}
 contains the integer $k = 1$. Therefore
with probability at least  $1 - 2e^{- 2 n}$ one has
\[
{\cal R}_n({\cal E}_p )
\geq
\frac{1}{2} \,  \min\Bigg\{ r_{1, p}, \,
\frac{S_N}{8 \, \sqrt{n}}  \Bigg\}
\geq \frac{1}{2} \,  \min\{ \sigma_{1}, A_N \} \ge \frac{\sigma_1}{2},
\]
where we used that $n \le f(N)=S_N^2 / (64 A_N^2)$ for the intermediate inequality.
This completes the proof.
\qed

\bigskip 

\noindent 
{\bf Acknowledgement. }  Parts of this project were carried out while the second- and fourth-named authors were in residence at the Hausdorff Research Institute for Mathematics in Winter 2024, during visits of the second-named author to Universit{\'e} Gustave Eiffel in December 2024 and in May–June 2025, and during a visit of the first-named author to the University of Alberta in January 2026. The authors are grateful for the financial support and the excellent working conditions provided by all these institutions.

\bibliographystyle{abbrv}
\bibliography{biblio}

@article{Paouris-Tikhomirov-Valettas-paper,
	author = {Paouris, Grigoris and Tikhomirov, Konstantin and Valettas, Petros},
	title = {Hypercontractivity and lower deviation estimates in normed spaces},
	fjournal = {The Annals of Probability},
	journal = {Ann. Probab.},
	issn = {0091-1798},
	volume = {50},
	number = {2},
	pages = {688--734},
	year = {2022},
	language = {English},
	doi = {10.1214/21-AOP1543},
	url = {projecteuclid.org/journals/annals-of-probability/volume-50/issue-2/Hypercontractivity-and-lower-deviation-estimates-in-normed-spaces/10.1214/21-AOP1543.full},
	zbMATH = {7512874},
	Zbl = {1496.46006}
}

@article{Klartag-Vershynin-paper,
	author = {Klartag, B. and Vershynin, R.},
	title = {Small ball probability and {Dvoretzky}'s {Theorem}},
	fjournal = {Israel Journal of Mathematics},
	journal = {Isr. J. Math.},
	issn = {0021-2172},
	volume = {157},
	pages = {193--207},
	year = {2007},
	language = {English},
	doi = {10.1007/s11856-006-0007-1},
	zbMATH = {5145475},
	Zbl = {1120.46003}
}

@book{Carl-book,
 author = {Carl, Bernd and Stephani, Irmtraud},
 title = {Entropy, compactness and the approximation of operators},
 fseries = {Cambridge Tracts in Mathematics},
 series = {Camb. Tracts Math.},
 issn = {0950-6284},
 volume = {98},
 isbn = {0-521-33011-4},
 year = {1990},
 publisher = {Cambridge etc.: Cambridge University Press},
 language = {English},
 zbMATH = {44592},
 Zbl = {0705.47017}
}

@article{HPS,
	author = {Hinrichs, Aicke and Prochno, Joscha and Sonnleitner, Mathias},
	title = {Random sections of {{\(\ell_p\)}}-ellipsoids, optimal recovery and {Gelfand} numbers of diagonal operators},
	fjournal = {Journal of Approximation Theory},
	journal = {J. Approx. Theory},
	issn = {0021-9045},
	volume = {293},
	pages = {19},
	year = {2023},
	language = {English},
	doi = {10.1016/j.jat.2023.105919},
	zbMATH = {7720797},
	Zbl = {1531.52005}
}

@article{HKNPU,
	author = {Hinrichs, Aicke and Krieg, David and Novak, Erich and Prochno, Joscha and Ullrich, Mario},
	title = {Random sections of ellipsoids and the power of random information},
	fjournal = {Transactions of the American Mathematical Society},
	journal = {Trans. Am. Math. Soc.},
	issn = {0002-9947},
	volume = {374},
	number = {12},
	pages = {8691--8713},
	year = {2021},
	language = {English},
	doi = {10.1090/tran/8502},
	zbMATH = {7618813},
	Zbl = {1501.60008}
}

@article{Gordon1985,
	author = {Gordon, Yehoram},
	title = {Some inequalities for {Gaussian} processes and applications},
	fjournal = {Israel Journal of Mathematics},
	journal = {Isr. J. Math.},
	issn = {0021-2172},
	volume = {50},
	pages = {265--289},
	year = {1985},
	language = {English},
	doi = {10.1007/BF02759761},
	zbMATH = {4084666},
	Zbl = {0663.60034}
}

@misc{Gordon1988,
	author = {Gordon, Y.},
	title = {On {Milman}'s inequality and random subspaces which escape through a mesh in {{\({\mathbb{R}}^ n\)}}},
year = {1988},
language = {English},
howpublished = {Geometric aspects of functional analysis, {Isr}. {Semin}. 1986-87, {Lect}. {Notes} {Math}. 1317, 84-106 (1988)},
doi = {10.1007/bfb0081737},
zbMATH = {4061904},
Zbl = {0651.46021}
}

@article{Gordon-ann-88,
 author = {Gordon, Yehoram},
 title = {Gaussian processes and almost spherical sections of convex bodies},
 fjournal = {The Annals of Probability},
 journal = {Ann. Probab.},
 issn = {0091-1798},
 volume = {16},
 number = {1},
 pages = {180--188},
 year = {1988},
 language = {English},
 doi = {10.1214/aop/1176991893},
 zbMATH = {4042961},
 Zbl = {0639.60046}
}

@article{Chevet1977,
   author = {Chevet, S.},
   title = {S{\'e}ries de variables al{\'e}atoires {Gaussiennes} {\`a} valeurs dans {$E\otimes_\varepsilon F$}. {Application} aux produits d'espaces de {Wiener} abstraits},
fjournal = {Semin. {Geom}. des {Espaces} de {Banach}, {Ec}. polytech., {Cent}. {Math}., 1977-1978, {Expose} {No}. 19, 1-15 (1978)},
journal = {Semin. {Geom}. des {Espaces} de {Banach}, {Ec}. polytech., {Cent}. {Math}., 1977-1978, {Expose} {No}. 19, 1-15 (1978)},
year = {1978},
language = {French},
url = {https://eudml.org/doc/109175},
zbMATH = {3613881},
Zbl = {0395.60004}
}

@article{RV2008,
	author = {Rudelson, Mark and Vershynin, Roman},
	title = {The {Littlewood}-{Offord} problem and invertibility of random matrices},
	fjournal = {Advances in Mathematics},
	journal = {Adv. Math.},
	issn = {0001-8708},
	volume = {218},
	number = {2},
	pages = {600--633},
	year = {2008},
	language = {English},
	doi = {10.1016/j.aim.2008.01.010},
	zbMATH = {5268159},
	Zbl = {1139.15015}
}

@book{Ledoux,
	author = {Ledoux, Michel},
	title = {The concentration of measure phenomenon},
	fseries = {Mathematical Surveys and Monographs},
	series = {Math. Surv. Monogr.},
	issn = {0076-5376},
	volume = {89},
	isbn = {0-8218-2864-9},
	year = {2001},
	publisher = {Providence, RI: American Mathematical Society (AMS)},
	language = {English},
	zbMATH = {1674764},
	Zbl = {0995.60002}
}

@article{PajorTomczak,
	author = {Pajor, Alain and Tomczak-Jaegermann, Nicole},
	title = {Subspaces of small codimension of finite-dimensional {Banach} spaces},
	fjournal = {Proceedings of the American Mathematical Society},
	journal = {Proc. Am. Math. Soc.},
	issn = {0002-9939},
	volume = {97},
	pages = {637--642},
	year = {1986},
	language = {English},
	doi = {10.2307/2045920},
	zbMATH = {4011192},
	Zbl = {0623.46008}
}

@book{MS-1200,
	author = {Milman, Vitali D. and Schechtman, Gideon},
	title = {Asymptotic theory of finite dimensional normed spaces. {With} an appendix by {M}. {Gromov}: {Isoperimetric} inequalities in {Riemannian} manifolds},
	fseries = {Lecture Notes in Mathematics},
	series = {Lect. Notes Math.},
	issn = {0075-8434},
	volume = {1200},
	year = {1986},
	publisher = {Springer, Cham},
	language = {English},
	doi = {10.1007/978-3-540-38822-7},
	zbMATH = {3979763},
	Zbl = {0606.46013}
}

@book{NW-1,
	author = {Novak, Erich and Wo{\'z}niakowski, Henryk},
	title = {Tractability of multivariate problems. {Volume} {I}: {Linear} information},
	fseries = {EMS Tracts in Mathematics},
	series = {EMS Tracts Math.},
	volume = {6},
	isbn = {978-3-03719-026-5},
	year = {2008},
	publisher = {Z{\"u}rich: European Mathematical Society (EMS)},
	language = {English},
	doi = {10.4171/026},
	zbMATH = {5320563},
	Zbl = {1156.65001}
}

@book{NW-2,
	author = {Novak, Erich and Wo{\'z}niakowski, Henryk},
	title = {Tractability of multivariate problems. {Volume} {II}: {Standard} information for functionals.},
	fseries = {EMS Tracts in Mathematics},
	series = {EMS Tracts Math.},
	volume = {12},
	isbn = {978-3-03719-084-5},
	year = {2010},
	publisher = {Z{\"u}rich: European Mathematical Society (EMS)},
	language = {English},
	doi = {10.4171/084},
	zbMATH = {5730459},
	Zbl = {1241.65025}
}

@book{NW-3,
	author = {Novak, Erich and Wo{\'z}niakowski, Henryk},
	title = {Tractability of multivariate problems. {Volume} {III}: {Standard} information for operators},
	fseries = {EMS Tracts in Mathematics},
	series = {EMS Tracts Math.},
	volume = {18},
	isbn = {978-3-03719-116-3},
	year = {2012},
	publisher = {Z{\"u}rich: European Mathematical Society (EMS)},
	language = {English},
	doi = {10.4171/116},
	zbMATH = {6100666},
	Zbl = {1359.65003}
}

@misc{M1985,
	author = {Milman, V. D.},
	title = {Random subspaces of proportional dimension of finite dimensional normed spaces: {Approach} through the isoperimetric inequality},
	year = {1985},
	language = {English},
	howpublished = {Banach spaces, {Proc}. {Conf}., {Columbia}/{Mo}. 1984, {Lect}. {Notes} {Math}. 1166, 106-115 (1985).},
	doi = {10.1007/bfb0074700},
	zbMATH = {3944537},
	Zbl = {0588.46013}
}

@article{GianMil1997,
	author = {Giannopoulos, A. A. and Milman, V. D.},
	title = {On the diameter of proportional sections of a symmetric convex body},
	fjournal = {IMRN. International Mathematics Research Notices},
	journal = {Int. Math. Res. Not.},
	issn = {1073-7928},
	volume = {1997},
	number = {1},
	pages = {5--19},
	year = {1997},
	language = {English},
	doi = {10.1155/S1073792897000020},
	zbMATH = {1001426},
	Zbl = {0877.52009}
}

@article{GiaMT2005,
	author = {Giannopoulos, Apostolos and Milman, Vitali D. and Tsolomitis, Antonis},
	title = {Asymptotic formulas for the diameter of sections of symmetric convex bodies},
	fjournal = {Journal of Functional Analysis},
	journal = {J. Funct. Anal.},
	issn = {0022-1236},
	volume = {223},
	number = {1},
	pages = {86--108},
	year = {2005},
	language = {English},
	doi = {10.1016/j.jfa.2004.10.006},
	zbMATH = {2183821},
	Zbl = {1088.46006}
}

@article{GiaM1998,
	author = {Giannopoulos, A. A. and Milman, V. D.},
	title = {Mean width and diameter of proportional sections of a symmetric convex body},
	fjournal = {Journal f{\"u}r die Reine und Angewandte Mathematik},
	journal = {J. Reine Angew. Math.},
	issn = {0075-4102},
	volume = {497},
	pages = {113--139},
	year = {1998},
	language = {English},
	doi = {10.1515/crll.1998.036},
	zbMATH = {1129236},
	Zbl = {0896.52007}
}

@incollection{LT2000,
	author = {Litvak, A. E. and Tomczak-Jaegermann, N.},
	title = {Random aspects of high-dimensional convex bodies},
	booktitle = {Geometric aspects of functional analysis. Proceedings of the Israel seminar (GAFA) 1996--2000},
	isbn = {3-540-41070-8},
	pages = {169--190},
	year = {2000},
	publisher = {Berlin: Springer},
	language = {English},
	doi = {10.1007/bfb0107214},
	zbMATH = {1574604},
	Zbl = {0986.52003}
}

@book{Pisier1989,
	author = {Pisier, Gilles},
	title = {The volume of convex bodies and {Banach} space geometry},
	fseries = {Cambridge Tracts in Mathematics},
	series = {Camb. Tracts Math.},
	issn = {0950-6284},
	volume = {94},
	isbn = {0-521-36465-5},
	year = {1989},
	publisher = {Cambridge etc.: Cambridge University Press},
	language = {English},
	zbMATH = {194093},
	Zbl = {0698.46008}
}

@book{AGM,
	author = {Artstein-Avidan, Shiri and Giannopoulos, Apostolos and Milman, Vitali D.},
	title = {Asymptotic geometric analysis. {I}},
	fseries = {Mathematical Surveys and Monographs},
	series = {Math. Surv. Monogr.},
	issn = {0076-5376},
	volume = {202},
	isbn = {978-1-4704-2193-9},
	year = {2015},
	publisher = {Providence, RI: American Mathematical Society (AMS)},
	language = {English},
	doi = {10.1090/surv/202},
	zbMATH = {6457742},
	Zbl = {1337.52001}
}

@book{DTU-CRM-Barcelona-2018,
	author = {D{\~u}ng, Dinh and Temlyakov, Vladimir and Ullrich, Tino},
	editor = {Tikhonov, Sergey},
	title = {Hyperbolic cross approximation. {Lecture} notes given at the courses on constructive approximation and harmonic analysis, {Barcelona}, {Spain}, {May} 30 -- {June} 3, 2016},
	fseries = {Advanced Courses in Mathematics -- CRM Barcelona},
	series = {Adv. Courses in Math. -- CRM Barc.},
	issn = {2297-0304},
	isbn = {978-3-319-92239-3; 978-3-319-92240-9},
	year = {2018},
	publisher = {Cham: Birkh{\"a}user},
	language = {English},
	doi = {10.1007/978-3-319-92240-9},
	zbMATH = {6967896},
	Zbl = {1414.41001}
}

@misc{KU2026-ActaNumerica,
	author = {Krieg, David and Ullrich, Mario},
	title = {Approximation of {Functions}: {Optimal} {Sampling} and {Complexity}},
	year = {2026},
	howpublished = {Preprint, {arXiv}:2602.02066 [math.{NA}] (2026)},
	url = {https://arxiv.org/abs/2602.02066},
	arXiv = {arXiv:2602.02066}
}

@article{LitvPajorTom,
	author = {Litvak, A. E. and Pajor, A. and Tomczak-Jaegermann, N.},
	title = {Diameters of sections and coverings of convex bodies},
	fjournal = {Journal of Functional Analysis},
	journal = {J. Funct. Anal.},
	issn = {0022-1236},
	volume = {231},
	number = {2},
	pages = {438--457},
	year = {2006},
	language = {English},
	doi = {10.1016/j.jfa.2005.06.013},
	zbMATH = {5019837},
	Zbl = {1092.52004}
}

@article{M-Dvor,
 author = {Mil'man, V. D.},
 title = {New proof of the theorem of {A}. {Dvoretzky} on intersections of convex bodies},
 fjournal = {Functional Analysis and its Applications},
 journal = {Funct. Anal. Appl.},
 issn = {0016-2663},
 volume = {5},
 pages = {288--295},
 year = {1972},
 language = {English},
 doi = {10.1007/BF01086740},
 zbMATH = {3379583},
 Zbl = {0239.46018}
}

@book{NTJ-book,
 author = {Tomczak-Jaegermann, Nicole},
 title = {Banach-{Mazur} distances and finite-dimensional operator ideals},
 fseries = {Pitman Monographs and Surveys in Pure and Applied Mathematics},
 series = {Pitman Monogr. Surv. Pure Appl. Math.},
 issn = {0269-3666},
 volume = {38},
 isbn = {0-582-01374-7; 0-470-20982-8},
 year = {1989},
 publisher = {Harlow: Longman Scientific \&| Technical; New York: John Wiley \&| Sons, Inc.},
 language = {English},
 zbMATH = {194266},
 Zbl = {0721.46004}
}

@article{EmMil,
	author = {Milman, Emanuel and Yifrach, Yuval},
	title = {Regular random sections of convex bodies and the random quotient-of-subspace theorem},
	fjournal = {Journal of Functional Analysis},
	journal = {J. Funct. Anal.},
	issn = {0022-1236},
	volume = {281},
	number = {7},
	pages = {22},
	note = {Id/No 109133},
	year = {2021},
	language = {English},
	doi = {10.1016/j.jfa.2021.109133},
	zbMATH = {7367927},
	Zbl = {1476.46017}
}

@article{GG1984,
	author = {Garnaev, A. Yu. and Gluskin, E. D.},
	title = {On widths of the {Euclidean} ball},
	fjournal = {Soviet Mathematics. Doklady},
	journal = {Sov. Math., Dokl.},
	issn = {0197-6788},
	volume = {30},
	pages = {200--204},
	year = {1984},
	language = {English},
	zbMATH = {3944477},
	Zbl = {0588.41022}
}

@article{K77,
	author = {Kasin, B. S.},
	title = {Diameters of some finite-dimensional sets and classes of smooth functions},
	fjournal = {Mathematics of the USSR. Izvestiya},
	journal = {Math. USSR, Izv.},
	issn = {0025-5726},
	volume = {11},
	pages = {317--333},
	year = {1977},
	language = {English},
	doi = {10.1070/IM1977v011n02ABEH001719},
	zbMATH = {3589184},
	Zbl = {0378.46027}
}

@incollection{HKNPU-Survey,
	author = {Hinrichs, Aicke and Krieg, David and Novak, Erich and Prochno, Joscha and Ullrich, Mario},
	title = {On the power of random information},
	booktitle = {Multivariate algorithms and information-based complexity},
	isbn = {978-3-11-063311-5; 978-3-11-063546-1},
	pages = {43--64},
	year = {2020},
	publisher = {Berlin: De Gruyter},
	language = {English},
	doi = {10.1515/9783110635461-004},
	zbMATH = {7224895},
	Zbl = {1523.65110}
}

\address
\end{document}